\documentclass[]{amsproc}

\numberwithin{equation}{section}
\usepackage{latexsym}
\usepackage{amsmath}
\usepackage{amssymb}
\usepackage{amsfonts}
\usepackage{verbatim}
\usepackage{mathrsfs}
\usepackage[latin1]{inputenc}
\usepackage{color}
\usepackage[colorlinks,citecolor=blue,urlcolor=blue]{hyperref}
\usepackage{subfigure}
\usepackage{graphicx}
\usepackage{amssymb}
\usepackage{mathrsfs}
\usepackage{float}
\usepackage{dsfont}
 \usepackage{geometry}
\allowdisplaybreaks \allowdisplaybreaks[4]

\newtheorem{theorem}{Theorem}[section]
\newtheorem{remark}{Remark}[section]

\newtheorem{definition}{Definition}[section]
\newtheorem{proposition}{Proposition}[section]
\newtheorem{lemma}{Lemma}[section]

\allowdisplaybreaks \allowdisplaybreaks[4]

\begin{document}

\title[Modified Milstein scheme for SDEs driven by fBm]
{Optimal convergence rate of modified Milstein scheme\\ for SDEs with rough fractional diffusions}

\author{Chuying Huang}
\address{College of Mathematics and Informatics \& FJKLMAA, Fujian Normal University, Fuzhou 350117, PR China}
\curraddr{}
\email{huangchuying@fjnu.edu.cn; huangchuying@lsec.cc.ac.cn}




\date{\today}

\dedicatory{}

\begin{abstract}
			We combine the rough path theory and stochastic backward error analysis to develop a new framework for error analysis on numerical schemes. Based on our approach, we prove that the almost sure convergence rate of the modified Milstein scheme for stochastic differential equations driven by multiplicative multidimensional fractional Brownian motion with Hurst parameter $H\in(\frac14,\frac12)$ is $(2H-\frac12)^-$ for sufficiently smooth coefficients, which is optimal in the sense that it is consistent with the result of the corresponding implementable approximation of the L\'evy area of fractional Brownian motion. Our result gives a positive answer to the conjecture proposed in \cite{Deya} for the case $H\in(\frac13,\frac12)$, and reveals for the first time that numerical schemes constructed by a second-order Taylor expansion do converge for the case $H\in(\frac14,\frac13]$.
\end{abstract}

\maketitle


	\section{Introduction}
In this article, we study the stochastic differential equation (SDE) driven by multiplicative multidimensional fractional Brownian motion (fBm)
\begin{align}\label{sde}
	\left\{
	\begin{aligned}
		{\rm d}Y_{t}&=\sigma(Y_{t}){\rm d}B_{t},\quad  t\in(0,T],\\
		Y_{0}&=z\in \mathbb{R}^m,
	\end{aligned}
	\right.
\end{align}
where $\sigma=(\sigma_1,\cdots,\sigma_d):\mathbb{R}^m\rightarrow L(\mathbb{R}^d,\mathbb{R}^m)$, and $B=(B^1,\cdots,B^d)$ is a $d$-dimensional fractional Brownian motion with Hurst parameter $H\in(\frac14,\frac12)$. More precisely, $B$ is a continuous centered Gaussian process characterized by the covariance function
\begin{align*}
	\mathbb{E}\Big[B^i_tB^j_s\Big]=\frac{1}{2}\Big[t^{2H}+s^{2H}-|t-s|^{2H}\Big]\mathds{1}_{\{i=j\}},\quad s,t\in [0,T],\quad i,j=1,\cdots,d.
\end{align*}
The properties for self-similarity, stationary increments and short-range dependence of fBm with Hurst parameter $H\in(\frac14,\frac12)$ lead to considerable practical applications of SDE \eqref{sde} such as the models for interest rates,  stochastic oscillators, circuit simulations, flows in porous media and so on; see e.g.  \cite{CHL17,Denk2007,ER19,FLW20,FZ17,HHW18,JR16,fbm}.

Under this setting, since fBm is not a martingale and the exponent of H\"older continuity of sample paths is $H^-<\frac12$, we interprete the SDE in the sense of rough path developed in \cite{MHFriz,Friz,Gubinelli,Lyons}, instead of the stochastic  It\^o integral for the case $H=\frac12$ or the pathwise fractional calculus for the case $H\in(\frac12,1)$. Based on the rough path theory, the well-posedness and robustness of pathwise solutions of equations driven by signals with $\frac1p$-H\"older regularity is established by smoothing the rough driving signal and applying the Taylor expansion up to $[p]$th-order, where $[p]$ is the integer part of $p\ge 1$. At the level of theoretical analysis, the robustness of solutions with respect to driving signals is fundamental for researches about the density and ergodicity of SDEs driven by non-Markovian stochastic processes \cite{Cass10,Cass15,MH13}. 
At the level of numerical approximation, it suggests intuitively that schemes constructed by a $[p]$th-order Taylor expansion converge for equations driven by signals with $\frac1p$-H\"older regularity.    	For instance, schemes constructed by a second-order Taylor expansion converge for the case $H\in(\frac13,\frac12)$ and schemes constructed by a third-order Taylor expansion converge for the case $H\in(\frac14,\frac13]$. This motivates us to investigate the optimal convergence rate of numerical schemes, to prove which it needs to develop new strategies since the probabilistic properties of fBm are essentially different from those of standard Brownian motion.

Consider a simple equation
\begin{align*}
	\left\{
	\begin{aligned}
		{\rm d}X^1_{t}&={\rm d}B^1_{t},\quad t\in(0,T],\\
		{\rm d}X^2_{t}&=X^1_t{\rm d}B^2_{t},\quad X^1_0=X^2_0=0,
	\end{aligned}
	\right.
\end{align*}
whose exact solution 
$$X^1_t=B^1_t,\quad X^2_t=\int_{0}^{t}\int_{0}^{u}{\rm d}B^1_{v}{\rm d}B^2_{u},$$
shows the L\'evy area of $B^1$ and $B^2$.
Based on a uniform gird $t_k=kh$, $n\in\mathbb{N}_+$, $h=\frac{T}{n}$, a second-order Taylor expansion leads to the Milstein scheme
\begin{align*}
	Z^{n,2}_{t_{k+1}}=&Z^{n,2}_{t_{k}}+\int_{t_k}^{t_{k+1}}Z^{n,1}_{t_k}{\rm d}B^2_{u}
	+\int_{t_k}^{t_{k+1}}\int_{t_k}^{u}{\rm d}B^1_{v}{\rm d}B^2_{u}\\
	=&Z^{n,2}_{t_{k}}+B^1_{t_k}\big(B^2_{t_{k+1}}-B^2_{t_{k}}\big)+\int_{t_k}^{t_{k+1}}\int_{t_k}^{u}{\rm d}B^1_{v}{\rm d}B^2_{u},\quad k=0,\cdots,n-1,
\end{align*}
which includes iterated integrals of fBm. However, due to the dependency of increments of fBm, the simulation for iterated integrals is rather difficult. One implementable method for numerical simulation is to substitute $(B^1,B^2)$ by a piecewise linear interpolation with time step size $h$ to construct a modified version
\begin{align}\label{Z}
	\tilde{Z}^{n,2}_{t_{k+1}}=&\tilde{Z}^{n,2}_{t_{k}}+B^1_{t_k}\big(B^2_{t_{k+1}}-B^2_{t_{k}}\big)+\frac12 \big(B^1_{t_{k+1}}-B^1_{t_{k}}\big)\big(B^2_{t_{k+1}}-B^2_{t_{k}}\big)\nonumber\\
	=&\tilde{Z}^{n,2}_{t_{k}}+\frac12 \big(B^1_{t_{k+1}}+B^1_{t_{k}}\big)\big(B^2_{t_{k+1}}-B^2_{t_{k}}\big),\quad k=0,\cdots,n-1.
\end{align}
In \cite{Neu10}, the authors prove for $H\in(\frac14,1)$ that the mean-square convergence rate of the modified Milstein scheme is $2H-\frac12$.
It is then natural to ask whether the same convergence rate of the modified Milstein scheme 
\begin{align}\label{scheme}
	Y^n_{t_{k+1}}=Y^n_{t_{k}}+\sigma(Y^n_{t_{k}})\Delta B_{{k+1}}
	+\frac12 \sigma'(Y^n_{t_{k}})\sigma(Y^n_{t_{k}}) \Delta B_{{k+1}}^{\otimes 2},\quad k=0,\cdots,n-1,
\end{align}
holds for the general SDE \eqref{sde},  where $\Delta B_{{k+1}}=B_{t_{k+1}}-B_{t_{k}}$.
If $H\in(\frac12,1)$, the question has already been solved in \cite{HHW} with the equation being understood by fractional calculus. If $H=\frac12$, it is actually a classical numerical conclusion for It\^o SDEs driven by standard Brownian motion; see e.g. \cite{MT04}. The contribution of this article is to fill this picture for $H\in(\frac14,\frac12)$.

The main difficulty in this topic is the low regularity of fBm as well as the solution of SDE \eqref{sde}. One alternative way is to use the Wong--Zakai approximation to turn the problem into a random ordinary differential equation. Then the local error of numerical schemes constructed by a second-order Taylor expansion is $3H^-$, that is, three times the exponent of H\"older continuity. Moreover, the robustness of solutions with respect to initial data in rough path theory derives that the pathwise global error is $(3H-1)^-$, which leads to a convergent situation for the case $H>\frac13$. To prove the optimal convergence rate, in this article, we combine the stochastic backward error analysis to further decompose the error between the numerical solution $Y^n$ and the exact solution of the Wong--Zakai approximation into two parts. The first part is the error between $Y^n$ and the exact solution of the associated truncated stochastic modified equation proposed in \cite{CHH}, which is proved to have arbitrary high order, by choosing the truncation number large enough. The second part is the error between the exact solutions of the associated truncated stochastic modified equation and the Wong--Zakai approximation. We rewrite the associated truncated stochastic modified equation as an equivalent equation driven by rough path lifted by a new stochastic process $\tilde{X}^n$ so that extra stochastic cancellation effects in error propagation are explored. Then together with the discrete sewing lemma and the robustness of solutions with respect to the driving signals in rough path theory, we obtain our main result for the whole situation $H\in(\frac14,\frac12)$.

\begin{theorem}\label{main-1}
	Let $1/4<H<1/2$ and $\gamma>3+\frac1H$. If $\sigma\in Lip^{\gamma}$,
	then for any sufficiently small $\epsilon>0$, there exists a random variable $G=G(\epsilon)$ independent of the time step size $h$ such that
	\begin{align*}
		\max_{k=1,\cdots,n}\big\|Y_{t_{k}}-Y^n_{t_{k}}\big\|\le Gh^{\min\{2H-1/2,H-1/\gamma\}-\epsilon},\quad a.s.,
	\end{align*}
	where $Y$ is the exact solution of \eqref{sde} in the sense of rough path and  $Y^n$ is the numerical solution given by the modified Milstein scheme \eqref{scheme}.
\end{theorem}
Here and in the rest of this article, we denote by $\|\cdot\|$ the Euclidean norm of $\mathbb{R}^d$. We use $C$ as a generic positive constant and $G$ as a generic positive random variable, which may be different from line to line and are all independent of the time step size $h$. 	
Based on Theorem \ref{main-1}, we obtain that the almost sure convergence rate of the modified Milstein scheme is $(2H-\frac12)^-$ for sufficiently smooth $\sigma$. It is optimal in the sense that the convergence rate corresponds to that of scheme \eqref{Z} for approximating the L\'evy area of fBm. This gives a positive answer to the optimal convergence rate conjecture in \cite{Deya} for the case $H\in(\frac13,\frac12)$. The more surprising thing is that the modified Milstein scheme is proved for the first time to be convergent for the case $H\in(\frac14,\frac13]$, and maintain the same convergence rate as that of third-order Taylor schemes. Indeed, our result holds for any implementable schemes constructed by a second-order Taylor expansion and for SDEs with drift terms. Moreover, we believe that our framework is applicable for SDEs driven by general rough signals.

From our procedure, we present a novel application for stochastic backward error analysis in stochastic forward error analysis.  For SDEs driven by standard Brownian motion, there are several types of weak stochastic modified equations with many applications, such as  constructing high weak order numerical schemes, studying invariant measures of numerical schemes and investigating the mathematical mechanism of stochastic symplectic methods for stochastic Hamiltonian systems. We refer to \cite{ACVZ12,2019BIT,DF12,K15BIT,K15IMA,S06,WHS16BIT,WXZ} for interested readers. 
As to numerical schemes constructed by a third-order Taylor expansion for SDEs driven by fBm, the remainder term of the local error between $Y$ and $Y^n$ has higher regularity, and the optimal strong convergence rate is obtained in \cite{MLMC16} for $H\in(\frac14,\frac12)$. For the modified Euler scheme, 
the optimal strong convergence rate when $H\in(\frac12,1)$ and the optimal almost sure convergence rate when $H\in(\frac13,\frac12)$ are proved in \cite{HuEuler} and \cite{AAP2019}, respectively. The optimal strong convergence rate of modified Milstein scheme and modified Euler scheme for rough case $H<\frac12$ is still an open problem. Besides, since the L\'evy area of fBm diverges if $H\le\frac14$, the well-posedness of SDEs in this case is unsolved.

The remainder of the article is structured as follows. The rough path theory and the stochastic backward error analysis are briefly introduced in Section \ref{sec-2}. The core framwork for proving Theorem \ref{main-1} is given in Section \ref{sec-3}. Technique estimates for the new process $\tilde{X}^n$ are proved in Section \ref{sec-4}.

\section{Preliminaries}\label{sec-2}
In Section \ref{sec-2.1}, we review the rough path theory developed in \cite{MHFriz,Friz,Gubinelli,Lyons} and illustrate the solution of SDE \eqref{sde} in the sense of rough path. In Section \ref{sec-2.2}, we introduce the stochastic backward error analysis proposed in \cite{CHH}, which constructs stochastic modified equations utilized in the proof of our main theorem.
\subsection{Rough path theory}\label{sec-2.1}

Let $[p]$ be the integer part of $p\ge 1$. Denote by  $\mathcal{C}^{1\text{-}var}\big([0,T];\mathbb{R}^d\big)$ the space of all continuous paths $x:[0,T]\rightarrow \mathbb{R}^d$ with bounded variation. We define the signature of $x\in\mathcal{C}^{1\text{-}var}\big([0,T];\mathbb{R}^d\big)$ by
\begin{align*}
	S_{[p]}(x)_t=\left(1,\int_{0\leq u_{1}\leq t}{\rm d} x_{u_{1}},\cdots,\int_{0\leq u_{1}<\cdots<u_{[p]}\leq t}{\rm d} x_{u_{1}}\otimes\cdots \otimes {\rm d} x_{u_{[p]}}\right)\in
	\mathbb{R}\oplus \Big(\oplus_{i=1}^{[p]}(\mathbb{R}^{d})^{\otimes i}\Big)
\end{align*}
and the space of all terminal values of signatures by
\begin{align*}
	G^{[p]}(\mathbb{R}^{d})=\Big\{S_{[p]}(x)_T: x\in\mathcal{C}^{1\text{-}var}\big([0,T];\mathbb{R}^d\big)\Big\}.
\end{align*} 
For any $\mathbf{a}=(\mathbf{a}^0,\mathbf{a}^1,\cdots,\mathbf{a}^{[p]}),\mathbf{b}=(\mathbf{b}^0,\mathbf{b}^1,\cdots,\mathbf{b}^{[p]})\in G^{[p]}(\mathbb{R}^{d})$,  the multiplication of $\mathbf{a}$ and $\mathbf{b}$ is given by
\begin{align*}
	\mathbf{a}\otimes \mathbf{b}=(\mathbf{c}^0,\mathbf{c}^1,\cdots,\mathbf{c}^{[p]})\in G^{[p]}(\mathbb{R}^{d}),~ \mathbf{c}^{i}=\sum_{k=0}^{i}\mathbf{a}^{k}\otimes\mathbf{b}^{i-k},\quad i=0,\cdots,[p].
\end{align*}
If $\mathbf{a}\otimes \mathbf{b}=(1,\mathbf{0},\cdots,\mathbf{0})$, then $\mathbf{b}$ is the inverse of $\mathbf{a}$ and denoted by $\mathbf{b}=\mathbf{a}^{-1}$.  In particular, 
\begin{align}\label{sig}
	(S_{[p]}(x)_s)^{-1}\otimes S_{[p]}(x)_t=\left(1,\int_{s\leq u_{1}\leq t}{\rm d} x_{u_{1}},\cdots,\int_{s\leq u_{1}<\cdots<u_{[p]}\leq t}{\rm d} x_{u_{1}}\otimes\cdots \otimes {\rm d} x_{u_{[p]}}\right).
\end{align}
Indeed, $G^{[p]}(\mathbb{R}^{d})$ is the free step-$[p]$ nilpotent Lie group of $\mathbb{R}^{d}$, for which we refer to \cite[Chapter 7]{Friz} for more details.

\begin{definition}
	If a continuous path $\mathbf{X}:[0,T]\rightarrow G^{[p]}(\mathbb{R}^{d})$ satisfies 
	\begin{align*}
		\|\mathbf{X}\|_{p\text{-}var;[0,T]}:=\sup_{\{t_k:k=0,\cdots,K\}\in \mathbb{D}([0,T])}\left(\sum_{k=0}^{K-1} \max_{i=1,\cdots,[p]}\big\|(\mathbf{X}^{-1}_{t_{k}}\otimes\mathbf{X}_{t_{k+1}})^i\big\|^\frac{p}{i}\right)^{\frac1p}
		<\infty,
	\end{align*}
	where $\mathbb{D}([0,T])$ is the set of all partitions of $[0,T]$,
	then we denote $\mathbf{X}\in \mathcal{C}^{p\text{-}var}\big([0,T];G^{[p]}(\mathbb{R}^{d})\big)$ and $\mathbf{X}$ is called a rough path.
	Furthermore,  a rough path $\mathbf{X}$ is of  H\"older-type, if 
	\begin{align*}
		\|\mathbf{X}\|_{\frac1p;[0,T]}:=\sup_{0\le s<t\le T}\frac{\max_{i=1,\cdots,[p]}\big\|(\mathbf{X}^{-1}_{s}\otimes\mathbf{X}_{t})^i\big\|^\frac1i}{|t-s|^{\frac1p}}< \infty.
	\end{align*}
\end{definition}

Since the rough path takes value in $G^{[p]}(\mathbb{R}^{d})$  instead of $\mathbb{R}^{d}$, it provides enough information to determine the well-posedness and robustness of solutions of equations with less regular driving signals. More precisely, 
for the rough differential equation (RDE)
\begin{align}\label{rde}
	\left\{
	\begin{aligned}
		{\rm d}Y_{t}&=\sigma(Y_{t}){\rm d}\mathbf{X}_{t},\quad t\in(0,T],\\
		Y_{0}&=z\in \mathbb{R}^{m},
	\end{aligned}
	\right.
\end{align}
the definition, well-posedness and robustness of its solution are stated as follows.

\begin{definition}\label{solution}
	Let $\mathbf{X} \in \mathcal{C}^{p\text{-}var}\big([0,T];G^{[p]}(\mathbb{R}^{d})\big)$. Suppose that for $n\in\mathbb{N}_+$ and $x^n\in\mathcal{C}^{1\text{-}var}\big([0,T];\mathbb{R}^{d}\big)$, 
	$y^{n}$ is the solution of 
	\begin{align}\label{equ-xn}
		\left\{
		\begin{aligned}
			{\rm d}y^{n}_{t}&=\sigma(y^{n}_{t}){\rm d}x^{n}_{t},\quad t\in(0,T],\\
			y^{n}_{0}&=z,
		\end{aligned}
		\right.
	\end{align}
	in the sense of Riemann--Stieltjes integral.
	If it holds that
	\begin{align}
		&\lim_{n\rightarrow \infty}	\sup_{0\le t\le T}\|y^{n}_t-Y_t\|=0,\label{yn}\\
		&\sup_{n\in \mathbb{N}_+}\|S_{[p]}(x^{n})\|_{p\text{-}var;[0,T]} < \infty,\label{xn1}\\
		&\lim_{n\rightarrow \infty} \sup_{0\leq s<t\leq T} \max_{i=1,\cdots,[p]}\big\|(S_{[p]}(x^{n})^{-1}_{t}\otimes S_{[p]}(x^{n})_{s}\otimes\mathbf{X}^{-1}_{s}\otimes\mathbf{X}_{t})^i\big\|^\frac1i=0,\label{xn2}
	\end{align}
	then $Y$ is called a solution of RDE \eqref{rde}.
\end{definition}

\begin{remark}
	For equations driven by paths in $\mathcal{C}^{1\text{-}var}\big([0,T];\mathbb{R}^d\big)$, the solution defined in Definition \ref{solution} equals to that in the sense of Riemann--Stieltjes integral.
\end{remark}

\begin{definition}\label{Lip}
	Let $\gamma>0$. Denote by $\lfloor \gamma \rfloor$ the largest integer such that  $\lfloor \gamma \rfloor<\gamma$. If $\sigma$ is $\lfloor \gamma \rfloor$-order differentiable, and there exists a  constant $C$ such that
	\begin{align*}
		\left\{
		\begin{aligned}
			&\big\|\sigma^{(k)}(y)\big\|\leq C,\quad k=0,\cdots, \lfloor \gamma \rfloor,~\ y \in \mathbb{R}^{m},\\
			&\big\|\sigma^{(\lfloor \gamma \rfloor)}(y_{1})-\sigma^{(\lfloor \gamma \rfloor)}(y_{2})\big\|\leq C \|y_{1}-y_{2}\|^{\gamma - \lfloor \gamma \rfloor},\quad y_{1},y_{2}\in \mathbb{R}^{m},
		\end{aligned}
		\right.
	\end{align*}
	where  $\sigma^{(k)}$ is the $k$th derivative of $\sigma$, then we say $\sigma \in Lip^{\gamma}$ and denote by $\|\sigma\|_{Lip^{\gamma}}$ the smallest constant satisfying the above inequalities.
\end{definition}

\begin{lemma}\label{lm-1}\rm{(\cite[Theorem 10.26]{Friz})}.
	Let $\mathbf{X}\in \mathcal{C}^{p\text{-}var}\big([0,T];G^{[p]}(\mathbb{R}^{d})\big)$ and $\{x^n\}_{n=1}^{\infty}\subseteq\mathcal{C}^{1\text{-}var}\big([0,T];\mathbb{R}^{d}\big)$ satisfy \eqref{xn1}-\eqref{xn2}.
	If $\sigma\in Lip^{\gamma}$, $\gamma>p$, 
	then RDE \eqref{rde} admits a unique solution. 
	Moreover, given another RDE
	\begin{align*}
		\left\{
		\begin{aligned}
			{\rm d}\tilde{Y}_{t}&=\tilde{\sigma}(\tilde{Y}_{t}){\rm d}\tilde{\mathbf{X}}_{t},\\
			\tilde{Y}_{0}&=\tilde{z},
		\end{aligned}
		\right.
	\end{align*}
	with $\tilde{\mathbf{X}} \in \mathcal{C}^{p\text{-}var}\big([0,T];G^{[p]}(\mathbb{R}^{d})\big)$ and $\tilde{\sigma} \in Lip^{\gamma}$, we have
	\begin{align*}
		\sup_{0\le t\le T}\|Y_t-\tilde{Y}_t\|\le C\exp\big\{C\omega(0,T)\big\}\Big[\|z-\tilde{z}\|+\|\sigma-\tilde{\sigma}\|_{Lip^{\gamma-1}}+U(\mathbf{X},\tilde{\mathbf{X}})\Big],
	\end{align*}	
	where $C=C(\gamma,p,T,\|\sigma\|_{Lip^{\gamma}},\|\tilde{\sigma}\|_{Lip^{\gamma}})$,  $\omega(s,t)=\|\mathbf{X}\|^p_{p\text{-}var;[s,t]}+\|\tilde{\mathbf{X}}\|^p_{p\text{-}var;[s,t]}$ and 
	\begin{align*}
		U(\mathbf{X},\tilde{\mathbf{X}})=\max_{i=1,\cdots,[p]} \sup_{0\leq s<t\leq T} \frac{\big\|(\mathbf{X}^{-1}_s\otimes\mathbf{X}_t)^i
			-(\tilde{\mathbf{X}}^{-1}_s\otimes\tilde{\mathbf{X}}_t)^i\big\|}{\omega(s,t)^{\frac{i}{p}}}.
	\end{align*}
\end{lemma}

For fBm with $H\in(\frac14,\frac12)$, let $\{x^n\}_{n=1}^{\infty}$ be a sequence of piecewise linear interpolations of $B$, i.e., 
\begin{align}\label{xh}
	x^{n}_{t}= x^{n}_{t_{k}}+(t-t_{k})h^{-1}\Delta B_{k+1},\quad h=\frac{T}{n}, \quad t \in (t_{k},t_{k+1}],\quad k=0,\cdots,n-1. 
\end{align}
Then $B$ is naturally lifted to a rough path $\mathbf{B}$ as a limit of $S_3(x^n)$, which is showed in the next lemma.

\begin{lemma}\label{lm-2}\rm{(\cite[Proposition 15.5 and Theorem 15.33]{Friz})}.
	Let $1/4<H<1/2$. Then for any $p>1/H$, the piecewise linear interpolations $\{x^n\}_{n=1}^{\infty}$ of $B$, which are defined in \eqref{xh}, satisfy
	\begin{align}\label{B1}
		\sup_{n\in \mathbb{N}_+}\|S_{3}(x^{n})\|_{\frac1p;[0,T]} < \infty,\quad a.s.
	\end{align}
	Moreover,  there exists a $\mathcal{C}^{p\text{-}var}\big([0,T];G^{[p]}(\mathbb{R}^{d})\big)$-valued random variable   $\mathbf{B}$ of  H\"older-type such that 
	\begin{align}\label{B2}
		\lim_{n\rightarrow \infty}\sup_{\{t_k:k=0,\cdots,K\}\in \mathbb{D}([0,T])}\left(\sum_{k=0}^{K-1} \max_{i=1,\cdots,3}\big\|(S_{3}(x^{n})^{-1}_{t_{k+1}}\otimes S_{3}(x^{n})_{t_{k}}\otimes\mathbf{B}^{-1}_{t_{k}}\otimes\mathbf{B}_{t_{k+1}})^i\big\|^\frac{p}{i}\right)^{\frac1p}=0,\quad a.s.
	\end{align}
\end{lemma}
As a consequence, for almost all sample paths, we interprete the solution of SDE \eqref{sde} as that of the following RDE
\begin{align}\label{rde1}
	\left\{
	\begin{aligned}
		{\rm d}Y_{t}&=\sigma(Y_{t}){\rm d}\mathbf{B}_{t},\quad t\in(0,T],\\
		Y_{0}&=z\in \mathbb{R}^{m}.
	\end{aligned}
	\right.
\end{align}
Furthermore, since \eqref{B1}-\eqref{B2} implies \eqref{xn1}-\eqref{xn2}, it is motivated to investigate the convergence of the solution of the Wong--Zakai approximation
\begin{align}\label{WZK}
	\left\{
	\begin{aligned}
		{\rm d}y^n_{t}&=\sigma(y^n_{t}){\rm d}x^n_{t},\quad  t\in(0,T],\\
		y^n_{0}&=z\in \mathbb{R}^m,
	\end{aligned}
	\right.
\end{align}
where $x^n$ is defined by \eqref{xh}. 

\begin{lemma}\label{lm-3}\rm{(\cite[Theorem 6 and Corollary 8]{Wzk})}.
	Let $1/4<H<1/2$ and $\gamma>1/H$.
	If $\sigma\in Lip^{\gamma}$,
	then for any sufficiently small $\epsilon>0$, there exists a random variable $G=G(\epsilon)$ independent of the time step $h$ such that
	\begin{align*}
		\sup_{0\le t\le T}\big\|Y_t-y^n_t\big\|\le Gh^{\min\{2H-1/2,H-1/\gamma\}-\epsilon},\quad a.s.,
	\end{align*}
	where  $Y$ and  $y^n$ are the exact solutions of SDE \eqref{sde} and the Wong--Zakai approximation \eqref{WZK}, respectively.
\end{lemma}

\subsection{Stochastic backward error analysis}\label{sec-2.2}
After applying certain numerical scheme to SDE \eqref{sde} and fixing the numerical solution $Y^n$, the goal of the stochastic backward error analysis is to evaluate the properties of $Y^n$ via a stochastic modified equation whose exact solution is much more close to $Y^n$ than $Y$. 
To this end, we assume that the numerical solution satisfies the expansion
\begin{align}\label{g}
	Y^n_{t_{k+1}}=Y^n_{t_{k+1}}+\sum_{|\alpha|=1}^\infty g_\alpha(Y^n_{t_{k}})(\Delta B^{1}_{k+1})^{\alpha_1}\cdots(\Delta B^{d}_{k+1})^{\alpha_d},
\end{align}
where $\alpha=(\alpha_1,\cdots,\alpha_d)\in(\mathbb{N}_+)^d$ and $|\alpha|=\alpha_1+\cdots+\alpha_d$. 
In \cite{CHH}, we construct the stochastic modified equation in the form of
\begin{align}\label{modified}
	\left\{
	\begin{aligned}
		\dot{\tilde{y}}^n_{t}
		&=\sum_{|\alpha|=1}^{\infty}f_\alpha(\tilde{y}^n_{t})h^{-1}(\Delta B^{1}_{k+1})^{\alpha_1}\cdots(\Delta B^{d}_{k+1})^{\alpha_d},\\
		\tilde{y}^n_{0}&=z,\quad t\in(t_k,t_{k+1}],\quad k=0,\cdots,n-1,
	\end{aligned}
	\right.
\end{align}
with the solution $\tilde{y}$ continuous on $[0,T]$. 
The coefficients $f_\alpha$, $|\alpha|=1$ are defined by an equivalent form of the Wong--Zakai approximation, which is
\begin{align}\label{wzk}
	\left\{
	\begin{aligned}
		\dot{y}^n_{t}
		&=\sum_{l=1}^{d}\sigma_l(y^n_{t}){\rm d}x^{n,l}_{t}=:\sum_{|\alpha|=1}f_\alpha(y^n_{t})h^{-1}(\Delta B^{1}_{k+1})^{\alpha_1}\cdots(\Delta B^{d}_{k+1})^{\alpha_d},\\
		y^n_{0}&=z,\quad t\in(t_k,t_{k+1}], ,\quad k=0,\cdots,n-1.
	\end{aligned}
	\right.
\end{align}
The coefficients $f_\alpha$, $|\alpha|\ge 2$ are determined by the iteration
\begin{align}\label{confj}
	f_\alpha(y)&=g_\alpha(y)-\sum_{i=2}^{|\alpha|}\frac{1}{i!}\sum_{(k^{i,1},\cdots,k^{i,i})\in O^\alpha_i}(D_{k^{i,1}}\cdots D_{k^{i,i-1}}f_{k^{i,i}})(y),\quad |\alpha|\ge 2,
\end{align}
where $$(D_{k^{i_1,i_2}}u)(y)=u'(y)f_{k^{i_1,i_2}}(y), \quad k^{i_1,i_2}=(k^{i_1,i_2}_1,\cdots,k^{i_1,i_2}_d)\in\mathbb{N}^{d},\quad |k^{i_1,i_2}|\ge 1,$$ and 
\begin{align*}
	O^\alpha_i=\Big\{ (k^{i,1},\cdots,k^{i,i}):  &k^{i,1},\cdots,k^{i,i}\in\mathbb{N}^{d},~|k^{i,1}|,\cdots,|k^{i,i}|\ge 1,\nonumber\\
	&k^{i,1}_l+\cdots+k^{i,i}_l=\alpha_l,~l=1,\cdots,d
	\Big\}.
\end{align*}
Then by comparing the Taylor expansion of $\dot{\tilde{y}}^n_{t_k}$ and $Y^n_{t_{k}}$, we have $$\dot{\tilde{y}}^n_{t_k}=Y^n_{t_{k}},\quad a.s.$$ 
Furthermore, given a fixed truncation number $N\in\mathbb{N}_+$, the truncated stochastic modified equation is defined by
\begin{align}\label{TME}
	\left\{
	\begin{aligned}
		\dot{\tilde{Y}}^n_{t}
		&=\sum_{|\alpha|=1}^{N}f_\alpha(\tilde{Y}^n_{t})h^{-1}(\Delta B^{1}_{k+1})^{\alpha_1}\cdots(\Delta B^{d}_{k+1})^{\alpha_d},\\
		\tilde{Y}^n_{0}&=z,\quad t\in(t_k,t_{k+1}],\quad k=0,\cdots,n-1,
	\end{aligned}
	\right.
\end{align}
with the solution $\tilde{Y}^n$ continuous on $[0,T]$. 

\begin{lemma}\label{lm-4}\rm{(\cite[Theorem 4.1]{CHH})}.
	Let $1/4<H<1/2$, $N>\frac1H-1$ and $\gamma>N+1/H$. If $\sigma\in Lip^{\gamma}$,
	then for any sufficiently small $\epsilon>0$, there exists a random variable $G=G(\epsilon)$ independent of the time step $h$ such that
	\begin{align*}
		\max_{k=1,\cdots,n}\big\|Y^n_{t_{k}}-\dot{\tilde{Y}}^n_{t_k}\big\|\le Gh^{(N+1)H-1-\epsilon},\quad a.s.,
	\end{align*}
	where  $Y^n$ is obtained by applying a numerical scheme to SDE \eqref{sde} such that  \eqref{g} holds,  and  $\dot{\tilde{Y}}^n$ is the exact solution of the associated truncated stochastic modified equation  \eqref{TME}.
\end{lemma}

\section{Proof of Theorem \ref{main-1}}\label{sec-3}
\begin{proof}		
	Noticing that the numerical solution $Y^n$ given by the modified Milstein scheme \eqref{scheme} has the formulation \eqref{g}, which is
	\begin{align}\label{gM}
		Y^n_{t_{k+1}}&=Y^n_{t_{k}}+\sum_{l=1}^d\sigma_l(Y^n_{t_{k}})\Delta B^l_{{k+1}}
		+\frac12 \sum_{l_1,l_2=1}^d\sigma_{l_1}'(Y^n_{t_{k}})\sigma_{l_2}(Y^n_{t_{k}}) \Delta B^{l_1}_{{k+1}}\Delta B^{l_2}_{{k+1}}\nonumber\\
		&=:Y^n_{t_{k+1}}+\sum_{|\alpha|=1}g_\alpha(Y^n_{t_{k}})(\Delta B^{1}_{k+1})^{\alpha_1}\cdots(\Delta B^{d}_{k+1})^{\alpha_d}
		+\sum_{|\alpha|=2}g_\alpha(Y^n_{t_{k}})(\Delta B^{1}_{k+1})^{\alpha_1}\cdots(\Delta B^{d}_{k+1})^{\alpha_d}\nonumber\\
		&=:Y^n_{t_{k+1}}+\sum_{|\alpha|=1}^\infty g_\alpha(Y^n_{t_{k}})(\Delta B^{1}_{k+1})^{\alpha_1}\cdots(\Delta B^{d}_{k+1})^{\alpha_d},
	\end{align}
	we decomposite the error by
	\begin{align*}
		\max_{k=1,\cdots,n}\big\|Y_{t_{k}}-Y^n_{t_{k}}\big\|\le\sup_{0\le t\le T}\big\|Y_t-y^n_t\big\|+\sup_{0\le t\le T}\big\|y^n_t-\tilde{Y}^n_t\big\|+
		\max_{k=1,\cdots,n}\big\|\tilde{Y}^n_{t_{k}}-Y^n_{t_{k}}\big\|,
	\end{align*}
	where $y^n$ is the exact solution of the Wong--Zakai approximation \eqref{WZK} and $\tilde{Y}^n$ is the exact solution of the truncated stochastic modified equation \eqref{TME} with truncation number $N=3$. 
	Applying the iteration \eqref{confj}, we obtain the coefficients of the truncated stochastic modified equation \eqref{TME}:
	\begin{align*}
		\left\{
		\begin{aligned}
			|\alpha|=1,~ \alpha_l=1:~f_\alpha(y)&=\sigma_l(y),\\
			|\alpha|=2:~f_\alpha(y)&=0, \\
			|\alpha|=3:~f_\alpha(y)&=-\frac{1}{6}\sum_{\alpha^1+\alpha^2+\alpha^3=\alpha}
			\Big[f''_{\alpha^1}(y)f_{\alpha^2}(y)f_{\alpha^3}(y)+f'_{\alpha^1}(y)f'_{\alpha^2}(y)f_{\alpha^3}(y)\Big].
		\end{aligned}
		\right.
	\end{align*}
	Based on Lemmas \ref{lm-3}-\ref{lm-4}, for any $\epsilon>0$, there exists a random variable $G=G(\epsilon)$ independent of $h$ such that 
	\begin{align*}
		\sup_{0\le t\le T}\big\|Y_t-y^n_t\big\|\le Gh^{\min\{2H-1/2,H-1/\gamma\}-\epsilon},\quad a.s.,
	\end{align*}
	and
	\begin{align*}
		\max_{k=1,\cdots,n}\big\|Y^n_{t_{k}}-\dot{\tilde{Y}}^n_{t_k}\big\|\le Gh^{4H-1-\epsilon},\quad a.s.
	\end{align*}
	In the following, we aim to prove  
	\begin{align}\label{key}
		\sup_{0\le t\le T}\big\|y^n_t-\tilde{Y}^n_t\big\|\le G h^{2H-\frac12-\epsilon},\quad a.s.
	\end{align}
	
	For $\alpha=(\alpha_1,\cdots,\alpha_d)\in(\mathbb{N}_+)^d$, denote by $\tilde{x}^{n,\alpha}$ an $\mathbb{R}$-valued stochastic process on $[0,T]$ such that 
	\begin{align}\label{x-alpha}
		\tilde{x}^{n,\alpha}_{t}=\tilde{x}^{n,\alpha}_{t_k}+(t-t_k)h^{-1}(\Delta B^{1}_{k+1})^{\alpha_1}\cdots(\Delta B^{d}_{k+1})^{\alpha_d},\quad h=\frac{T}{n},\quad t\in (t_k,t_{k+1}],\quad k=0,\cdots,n-1.
	\end{align}
	Note that if $|\alpha|=1$ with $\alpha_l=1$, then $\tilde{x}^{n,\alpha}$ is the $l$th component of the piecewise linear interpolation $x^n$ of $B$   defined in \eqref{xh}.
	We construct a new multidimensional stochastic process $\tilde{X}^n$ which satisfies that its first $d$ components are $\tilde{x}^{n,\alpha}$ for  $|\alpha|=1$ with $\alpha_l=1$, $l=1,\cdots,d$, and that the other components are $\tilde{x}^{n,\alpha}$ with $|\alpha|=3$. By means of this process, the truncated stochastic modified equation \eqref{TME} with $N=3$ is equivalent to an equation driven by $\tilde{X}^n$, i.e., 
	\begin{align}\label{step3-1}
		\left\{
		\begin{aligned}
			{\rm d}\tilde{Y}^n_{t}
			&=\sum_{|\alpha|=1,3}f_\alpha(\tilde{Y}^n_{t}){\rm d} \tilde{x}^{n,\alpha}_t=:V(\tilde{Y}^n_{t}){\rm d} \tilde{X}^{n}_t,\quad t\in(0,T],\\
			\tilde{Y}^n_{0}&=z.
		\end{aligned}
		\right.
	\end{align}
	Meanwhile, we rewrite the Wong--Zakai approximation \eqref{WZK} as
	\begin{align}\label{step3-2}
		\left\{
		\begin{aligned}
			{\rm d}y^n_{t}
			&=V(y^n_{t}){\rm d} \bar{X}^{n}_t,\quad t\in(0,T],\\
			y^n_{0}&=z,
		\end{aligned}
		\right.
	\end{align}
	where the dimension of $\bar{X}^n=(x^n,0,0...,0,0)$ is the same as that of $\tilde{X}^n$  and their first $d$  components are the same. Since $\tilde{X}^n$ and $\bar{X}^n$ both have bounded variations for almost all sample paths, equations \eqref{step3-1}-\eqref{step3-2} can be interpreted in the sense of RDEs driven by $S_{3}(\tilde{X}^n)$ and $S_{3}(\bar{X}^n)$, respectively. Moreover, according to Lemma \ref{lm-2}, we have for any $0<\beta <H$ that
	\begin{align*}
		\sup_{n\in \mathbb{N}_+}\|S_{3}(\bar{X}^{n})\|_{\beta;[0,T]} < \infty,\quad a.s.
	\end{align*}
	Together with Lemma \ref{lm-1}, a sufficient condition for \eqref{key} is
	\begin{align*}
		\max_{i=1,2,3} \sup_{0\leq s<t\leq T} \frac{\Big\|\big(S_{3}(\tilde{X}^n)^{-1}_s\otimes S_{3}(\tilde{X}^n)_t\big)^i
			-\big(S_{3}(\bar{X}^n)^{-1}_s\otimes S_{3}(\bar{X}^n)_t\big)^i\Big\|}{|t-s|^{i\beta}}\le Gh^{2H-1/2-\epsilon},\quad a.s.
	\end{align*}
	Based on the formula \eqref{sig}, it remains to require the following inequalities
	\begin{align*}
		&\sup_{0\leq s<t\leq T} \frac{\left|\int_{s}^{t}{\rm d}\tilde{x}^{n,\alpha}_{u_1}\right|}{|t-s|^{\beta}}\le Gh^{2H-1/2-\epsilon} ,\quad |\alpha|=3,\\
		&\sup_{0\leq s<t\leq T} \frac{\left|\int_{s}^{t}\int_{s}^{u_1}{\rm d}\tilde{x}_{u_2}^{n,\alpha^2}{\rm d}\tilde{x}^{n,\alpha^1}_{u_1}\right|}{|t-s|^{2\beta}}\le G h^{2H-1/2-\epsilon},\quad |\alpha^1|=1,~|\alpha^2|=3,\\
		&\sup_{0\leq s<t\leq T} \frac{\left|\int_{s}^{t}\int_{s}^{u_1}{\rm d}\tilde{x}_{u_2}^{n,\alpha^2}{\rm d}\tilde{x}^{n,\alpha^1}_{u_1}\right|}{|t-s|^{2\beta}}\le G h^{2H-1/2-\epsilon},\quad |\alpha^1|=3,~|\alpha^2|=1,\\
		&\sup_{0\leq s<t\leq T} \frac{\left|\int_{s}^{t}\int_{s}^{u_1}{\rm d}\tilde{x}_{u_2}^{n,\alpha^2}{\rm d}\tilde{x}^{n,\alpha^1}_{u_1}\right|}{|t-s|^{2\beta}}\le G h^{2H-1/2-\epsilon},\quad |\alpha^1|=|\alpha^2|=3,\\
		&\sup_{0\leq s<t\leq T} \frac{\left|\int_{s}^{t}\int_{s}^{u_1}\int_{s}^{u_2}{\rm d}\tilde{x}_{u_3}^{n,\alpha^3}{\rm d}\tilde{x}_{u_2}^{n,\alpha^2}{\rm d}\tilde{x}^{n,\alpha^1}_{u_1}\right|}{|t-s|^{3\beta}}\le G h^{2H-1/2-\epsilon},\quad |\alpha^1|=|\alpha^2|=3,~|\alpha^3|=1,\\
		&\sup_{0\leq s<t\leq T} \frac{\left|\int_{s}^{t}\int_{s}^{u_1}\int_{s}^{u_2}{\rm d}\tilde{x}_{u_3}^{n,\alpha^3}{\rm d}\tilde{x}_{u_2}^{n,\alpha^2}{\rm d}\tilde{x}^{n,\alpha^1}_{u_1}\right|}{|t-s|^{3\beta}}\le G h^{2H-1/2-\epsilon},\quad |\alpha^2|=|\alpha^3|=3,~|\alpha^1|=1,\\
		&\sup_{0\leq s<t\leq T} \frac{\left|\int_{s}^{t}\int_{s}^{u_1}\int_{s}^{u_2}{\rm d}\tilde{x}_{u_3}^{n,\alpha^3}{\rm d}\tilde{x}_{u_2}^{n,\alpha^2}{\rm d}\tilde{x}^{n,\alpha^1}_{u_1}\right|}{|t-s|^{3\beta}}\le G h^{2H-1/2-\epsilon},\quad |\alpha^1|=|\alpha^3|=3,~|\alpha^2|=1,\\
		&\sup_{0\leq s<t\leq T} \frac{\left|\int_{s}^{t}\int_{s}^{u_1}\int_{s}^{u_2}{\rm d}\tilde{x}_{u_3}^{n,\alpha^3}{\rm d}\tilde{x}_{u_2}^{n,\alpha^2}{\rm d}\tilde{x}^{n,\alpha^1}_{u_1}\right|}{|t-s|^{3\beta}}\le G h^{2H-1/2-\epsilon},\quad |\alpha^1|=3,~|\alpha^2|=|\alpha^3|=1,\\
		&\sup_{0\leq s<t\leq T} \frac{\left|\int_{s}^{t}\int_{s}^{u_1}\int_{s}^{u_2}{\rm d}\tilde{x}_{u_3}^{n,\alpha^3}{\rm d}\tilde{x}_{u_2}^{n,\alpha^2}{\rm d}\tilde{x}^{n,\alpha^1}_{u_1}\right|}{|t-s|^{3\beta}}\le G h^{2H-1/2-\epsilon},\quad |\alpha^2|=3,~|\alpha^1|=|\alpha^3|=1,\\
		&\sup_{0\leq s<t\leq T} \frac{\left|\int_{s}^{t}\int_{s}^{u_1}\int_{s}^{u_2}{\rm d}\tilde{x}_{u_3}^{n,\alpha^3}{\rm d}\tilde{x}_{u_2}^{n,\alpha^2}{\rm d}\tilde{x}^{n,\alpha^1}_{u_1}\right|}{|t-s|^{3\beta}}\le G h^{2H-1/2-\epsilon},\quad |\alpha^3|=3,~|\alpha^1|=|\alpha^2|=1,\\
		&\sup_{0\leq s<t\leq T} \frac{\left|\int_{s}^{t}\int_{s}^{u_1}\int_{s}^{u_2}{\rm d}\tilde{x}_{u_3}^{n,\alpha^3}{\rm d}\tilde{x}_{u_2}^{n,\alpha^2}{\rm d}\tilde{x}^{n,\alpha^1}_{u_1}\right|}{|t-s|^{3\beta}}\le G h^{2H-1/2-\epsilon},\quad |\alpha^1|=|\alpha^2|=|\alpha^3|=3,
	\end{align*}
	which are proved in Propositions \ref{prop-1}-\ref{prop-3} in the next section.
	Therefore, we conclude 
	\begin{align}\label{key1}
		\max_{k=1,\cdots,n}\big\|Y_{t_{k}}-Y^n_{t_{k}}\big\|\le Gh^{\min\{2H-1/2,H-1/\gamma\}-\epsilon},\quad a.s.,
	\end{align}
	where the random variable $G$ is independent of $h$.
\end{proof}

\begin{remark}
	If $\sigma\in Lip^{\gamma}_{{\rm loc}}$, the result in Theorem \ref{main-1} holds with a localization argument as long as the pathwise solution does not explode.
\end{remark}

\begin{remark}
	For numerical schemes constructed by a second-order Taylor expansion, such as the Crank--Nicolson scheme in \cite{CN2017} and the midpoint scheme in \cite{HHW18}, since the associated coefficients of the stochastic modified equation satisfy $f_\alpha(y)=0$, $|\alpha|=2$, we obtain immediately from the proof of Theorem \ref{main-1} that the convergence rate of these schemes are the same. In other words, the error estimate \eqref{key1} holds for these schemes.
\end{remark}

\begin{remark}
	Based on the fact that the regularity of drift terms is better than that of diffusion terms so that the convergence rates for multiple integrals $\tilde{X}^n$ showed in Propositions \ref{prop-1}-\ref{prop-3} do not decrease, Theorem \ref{main-1} can be extended to SDEs with drift terms. Moreover, the framework presented in the proof of Theorem \ref{main-1} is valid for SDEs driven by general signals, once convergence rates for multiple integrals $\tilde{X}^n$ and piecewise linear appromation $x^n$ are established. 
\end{remark}

\section{Technical estimates}\label{sec-4}
In this section, we first introduce several lemmas and then prove the estimates for multiple integrals of $\tilde{X}^n$ in Propositions \ref{prop-1}-\ref{prop-3}, which are crucial in the proof of the main theorem.

\begin{lemma}\label{lm1}\rm{(Besov--H\"older embedding; see e.g. \cite[Corollary A.2]{Wzk})}.
	Let $q>1$ and $\frac1q<\alpha<1$. For a continuous function $x:[0,T]\rightarrow \mathbb{R}^d$ and $0\le s<t\le T$, it holds that
	\begin{align*}
		\|x\|_{\alpha-\frac1q;[s,t]}:=\sup_{s\le u<v\le t}\frac{\|x_u-x_v\|}{|u-v|^{\alpha-\frac1q}}\le C(\alpha,q)\bigg(\int_{s}^{t}\int_{s}^{t}
		\frac{\|x_u-x_v\|^q}{|u-v|^{1+\alpha q}}
		{\rm d}u{\rm d}v\bigg)^{\frac1q}.
	\end{align*}
\end{lemma}

\begin{lemma}\label{lm2}
	Let $0<\alpha\le 1$.  If for all $q\ge 1$, a sequence of  $\mathbb{R}^d$-valued stochastic processes $\{g^n\}_{n=1}^\infty$ on $[0,T]$ satisfies
	\begin{align*}
		\sup_{0\le s<t\le T}\frac{\|g^n_t-g^n_s\|_{L^q(\Omega;\mathbb{R}^d)}}{|t-s|^\alpha}\le Cn^{-\gamma}
	\end{align*}
	with $C=C(q)$ independent of $n$, then for any $0<\delta<\alpha$ and $q\ge 1$, there exists a constant $C=C(\delta,\alpha,q,T)$ independent of $n$ such that 
	\begin{align*}
		\left\|\sup_{0\le s<t\le T}\frac{\|g^n_t-g^n_s\|}{|t-s|^{\alpha-\delta}}\right\|_{L^q(\Omega)}\le Cn^{-\gamma}.
	\end{align*}
	Moreover, for any $0<\epsilon<\gamma$, there exists a random variable $G=G(\epsilon,\delta,\alpha,T)$ independent of $n$ such that 
	\begin{align*}
		\sup_{0\le s<t\le T}\frac{\|g^n_t-g^n_s\|}{|t-s|^{\alpha-\delta}}\le Gn^{-(\gamma-\epsilon)}.
	\end{align*}	
\end{lemma}
\begin{proof}
	Let $q=\frac{2}{\delta}$ so that $\alpha-\delta=(\alpha-1/q)-1/q$. Based on Lemma \ref{lm1}, we have 
	\begin{align*}
		\mathbb{E}\Bigg[\bigg(\sup_{0\le s<t\le T}\frac{\|g^n_t-g^n_s\|}{|t-s|^{\alpha-\delta}}\bigg)^q\Bigg]= &
		\Big\|\|g^n\|_{(\alpha-\frac1q)-\frac1q;[0,T]}\Big\|_{L^q(\Omega)}^q\\
		\le &C\int_{0}^{T}\int_{0}^{T} \frac{\mathbb{E}\big[\|g^n_u-g^n_v\|^q\big]}{|u-v|^{1+(\alpha-1/q)q}}{\rm d}u{\rm d}v\\
		\le &C\int_{0}^{T}\int_{0}^{T} \frac{|u-v|^{\alpha q}n^{-\gamma q}}{|u-v|^{1+(\alpha-1/q)q}}dudv\\
		\le &Cn^{-\gamma q}.
	\end{align*}
	
	Taking $q>1$ such that $\epsilon q >1$,  we obtain
	\begin{align*}
		\mathbb{E}\bigg[\sup_{n\in \mathbb{N}}\|n^{(\gamma-\epsilon)q}g^n\|^q_{\alpha-\delta;[0,T]}\bigg]
		\le \sum_{n=1}^{\infty}\mathbb{E}\bigg[\|n^{(\gamma-\epsilon)q}g^n\|^q_{\alpha-\delta;[0,T]}\bigg]
		\le C\sum_{n=1}^{\infty} n^{-\epsilon q}\le C,
	\end{align*}
	which impies the conclusion.
\end{proof}

\begin{lemma}\label{lm3}
	Let $0<H<1$. Given positive integers $m,i,j,k$ satisfying $m\ge 1$, $1\le i\le d$ and $0\le j<k\le n=\frac{T}{h}$. We have that
	\begin{align*}
		\bigg\|      \sum_{l=j+1}^{k} (\Delta B^i_l)^m  \bigg\|_{L^2(\Omega)}&\le C h^{mH-1/2}|t_k-t_j|^{1/2},
		\quad if~m~is~odd,\\
		\bigg\|      \sum_{l=j+1}^{k}  (\Delta B^i_l)^m \bigg\|_{L^2(\Omega)}&\le Ch^{mH-1}|t_k-t_j|,
		~~~\quad\quad if~m~is~even,
	\end{align*}
	where $C=C(m)$ independent of $i,k,j,h$.
\end{lemma}
\begin{proof}
	Based on the asymptotical behavior of the covariance of increments of fBm given in \cite[Section 7.4]{NP12}, we apply \cite[Theorem 1]{BM83} to the random variables  $A_l=\frac{\Delta B^i_l}{{h^H}}$, $l=j+1,\cdots,k$, which have zero mean and unit variance, and obtain that 
	\begin{align*}
		\lim_{n\rightarrow \infty}\mathbb{E}\bigg[\Big(|t_k-t_j|^{-\frac12} h^{\frac12}\sum_{l=j+1}^{k}  \mathbf{H}_m( A_l)\Big)^2\bigg]=C,
	\end{align*}
	where $\mathbf{H}_m$ are $m$th monic Hermite polynomial. Noticing that
	if $m$ is odd, then the degrees of the terms in $\mathbf{H}_m(x)$ are all odd, then we have
	\begin{align*}
		\bigg\|      \sum_{l=j+1}^{k} (\Delta B^i_l)^m  \bigg\|_{L^2(\Omega)}\le C h^{mH-1/2}|t_k-t_j|^{1/2}.
	\end{align*}
	If $m$ is even, then the degrees of the terms in $\mathbf{H}_m(x)$ are all even. In this case, we combine 
	\begin{align*}
		\lim_{n\rightarrow \infty}\mathbb{E}\bigg[\Big(|t_k-t_j|^{-1} h\sum_{l=j+1}^{k} 1 \Big)^2\bigg]=C
	\end{align*}
	to get 
	\begin{align*}
		\bigg\|      \sum_{l=j+1}^{k}  (\Delta B^i_l)^m \bigg\|_{L^2(\Omega)}\le Ch^{mH-1}|t_k-t_j|,
	\end{align*}
	which finishes the proof.
\end{proof}

\begin{lemma}\label{lm4}
	For centered Gaussian random vectors $A$ and $\tilde{A}$, it holds that 
	\begin{align*}
		\mathbb{E}[A^2\tilde{A}^2]&=\mathbb{E}[A^2]\mathbb{E}[\tilde{A}^2]
		+2\big(\mathbb{E}[A\tilde{A}]\big)^2,\\
		\mathbb{E}[A^3\tilde{A}^3]&=6\big(\mathbb{E}[A\tilde{A}]\big)^3+9\mathbb{E}[A\tilde{A}] \mathbb{E}[A^2]\mathbb{E}[\tilde{A}^2].
	\end{align*}
\end{lemma}
\begin{proof}
	We refer to \cite[Lemma 3.7]{RV00SSR} for the first formula and to \cite[Proposition 2.3]{G03AOP} for the second one.
\end{proof}

\begin{lemma}\label{lm5}\rm{(Discrete sewing lemma; see e.g. \cite[Lemma 2.5]{AAP2019})}.
	Assume that $f_{t_j,t_k}$ is a function defined on $\{(t_j,t_k): 0\le t_j=jh<t_k=kh\le T\}$. If  $f_{t_j,t_{j+1}}=0$, $j=0,...,n-1$, and there exists a constant $\mu>1$ such that
	\begin{align*}
		\sup_{(t_j,t_k,t_l)}\frac{\big\|f_{t_j,t_l}-f_{t_j,t_k}-f_{t_k,t_l}\big\|}{|t_j-t_l|^\mu}\le C_0,
	\end{align*} 
	then 
	\begin{align*}
		\sup_{(t_j,t_k)}\frac{\|f_{t_j,t_k}\|}{|t_j-t_k|^\mu}\le C(\mu)C_0.
	\end{align*}
\end{lemma}

\begin{proposition}\label{prop-1}
	Under the assumptions  in Theorem \ref{main-1} and the definition of $\tilde{x}^{n,\alpha}$ in \eqref{x-alpha}, then for any $0<\delta<\frac12$ and $0<\epsilon<3H-\frac12$, there exists a random variable $G=G(\delta,\epsilon)$ independent of $h=\frac{T}{n}$ such that
	\begin{align*}
		\sup_{0\leq s<t\leq T} \frac{\left|\int_{s}^{t}{\rm d}\tilde{x}^{n,\alpha}_{u_1}\right|}{|t-s|^{\frac12-\delta}}\le Gh^{3H-1/2-\epsilon},\quad \alpha=(\alpha_1,\cdots,\alpha_d),\quad |\alpha|=3.
	\end{align*}
\end{proposition}
\begin{proof}
	According to Lemma \ref{lm2} and hypercontractivity property, it suffices to prove 
	\begin{align*}
		\big\|\tilde{x}^{n,\alpha}_t-\tilde{x}^{n,\alpha}_s\big\|_{L^2(\Omega)}\le Ch^{3H-\frac12}|t-s|^\frac12,\quad \alpha=(\alpha_1,\cdots,\alpha_d),\quad |\alpha|=3.
	\end{align*}
	
	For $t_j\le s<t\le t_{j+1}$,  we have from $|t-s|\le h$ that
	\begin{align*}
		\big\|\tilde{x}^{n,\alpha}_t-\tilde{x}^{n,\alpha}_s\big\|_{L^2(\Omega)}
		=\frac{t-s}{h}\bigg\|(\Delta B^{1}_{k+1})^{\alpha_1}\cdots(\Delta B^{d}_{k+1})^{\alpha_d}\bigg\|_{L^2(\Omega)}
		\le C\bigg(\frac{t-s}{h}\bigg)^\frac12 h^{3H}
		\le Ch^{3H-\frac12}|t-s|^\frac12.
	\end{align*}
	
	For $t_{j-1}<s\le t_j\le t_{k}\le t<t_{k+1}$, based on the previous estimate and the fact
	\begin{align*}
		\big|\tilde{x}^{n,\alpha}_t-\tilde{x}^{n,\alpha}_s\big|\le
		\big|\tilde{x}^{n,\alpha}_t-\tilde{x}^{n,\alpha}_{t_{k}}\big|+
		\big|\tilde{x}^{n,\alpha}_{t_k}-\tilde{x}^{n,\alpha}_{t_{j}}\big|+
		\big|\tilde{x}^{n,\alpha}_{t_j}-\tilde{x}^{n,\alpha}_s\big|, 
	\end{align*}
	it suffices to consider the case $s= t_j< t_k=t$. If there exists an integer $i\in\{1,\cdots,d\}$ such that $\alpha_i=3$, then Lemma \ref{lm3} leads to 
	\begin{align*}
		\big\|\tilde{x}^{n,\alpha}_t-\tilde{x}^{n,\alpha}_s\big\|_{L^2(\Omega)}
		= \bigg\|\sum_{l=j+1}^{k} \big(\Delta B^{i_1}_{l}\big) ^3\bigg\|_{L^2(\Omega)}
		\le Ch^{3H-\frac12}|t-s|^\frac12.
	\end{align*}
	If there exist three distinct integers $i_1,i_2,i_3\in\{1,\cdots,d\}$ such that $\alpha_{i_1}=\alpha_{i_2}=\alpha_{i_3}=1$, then Lemmas \ref{lm3}-\ref{lm4} produce
	\begin{align*}
		0\le &\big\|\tilde{x}^{n,\alpha}_t-\tilde{x}^{n,\alpha}_s\big\|^2_{L^2(\Omega)}\\
		=&\mathbb{E}\Bigg[ \Bigg(\sum_{l=j+1}^{k} \big(\Delta B^{i_1}_{l}\big)\big(\Delta B^{i_2}_{l}\big)\big(\Delta B^{i_3}_{l}\big) \Bigg)\Bigg( \sum_{r=j+1}^{k} \big(\Delta B^{i_1}_r\big)\big(\Delta B^{i_2}_r\big) \big(\Delta B^{i_3}_r\big)  \Bigg)\Bigg]\\
		=&\sum_{l=j+1}^{k}\sum_{r=j+1}^{k}\mathbb{E}\Big[  \big(\Delta B^{i_1}_{l}\big) \big(\Delta B^{i_1}_r\big)  \Big]\mathbb{E}\Big[  \big(\Delta B^{i_2}_{l}\big)  \big(\Delta B^{i_2}_r\big) \Big]\mathbb{E}\Big[  \big(\Delta B^{i_3}_{l}\big)  \big(\Delta B^{i_3}_r\big) \Big]\\
		=&\sum_{l=j+1}^{k}\sum_{r=j+1}^{k}\Big(\mathbb{E}\Big[  \big(\Delta B^{i_1}_{l}\big)  (\Delta B^{i_1}_r) \Big]\Big)^3\\
		=&\frac16\sum_{l=j+1}^{k}\sum_{r=j+1}^{k}\bigg(\mathbb{E}\Big[  \big(\Delta B^{i_1}_{l}\big) ^3 (\Delta B^{i_1}_r) ^3\Big] - 9
		\mathbb{E}\Big[  \big(\Delta B^{i_1}_{l}\big) (\Delta B^{i_1}_r)\Big]\mathbb{E}\Big[  \big(\Delta B^{i_1}_{l}\big)^2\Big] \mathbb{E}\Big[(\Delta B^{i_1}_r)^2\Big]  \bigg)\\
		=&\frac16\sum_{l=j+1}^{k}\sum_{r=j+1}^{k}\mathbb{E}\Big[  \big(\Delta B^{i_1}_{l}\big) ^3 (\Delta B^{i_1}_r) ^3\Big] - \frac32 h^{4H} |t-s|^{2H} \\
		\le & C \mathbb{E}\Bigg[ \bigg(\sum_{l=j+1}^{k} \big(\Delta B^{i_1}_{l}\big) ^3\bigg)^2\Bigg]\le Ch^{6H-1}|t-s|.
	\end{align*}
	Similarly, if there exist two distinct integers $i_1,i_2\in\{1,\cdots,d\}$ such that $\alpha_{i_1}=2$ and $\alpha_{i_2}=1$, then we have from Lemmas \ref{lm3}-\ref{lm4} that 
	\begin{align*}
		0\le &\big\|\tilde{x}^{n,\alpha}_t-\tilde{x}^{n,\alpha}_s\big\|^2_{L^2(\Omega)} \\
		=&\mathbb{E}\Bigg[ \Bigg(\sum_{l=j+1}^{k} \big(\Delta B^{i_1}_l\big)^2\big(\Delta B^{i_2}_l\big) \Bigg) \Bigg(\sum_{r=j+1}^{k} \big(\Delta B^{i_1}_r\big)^2\big(\Delta B^{i_2}_r\big)  \Bigg)\Bigg]\\
		=&\sum_{l=j+1}^{k}\sum_{r=j+1}^{k}\mathbb{E}\Big[  \big(\Delta B^{i_1}_l\big)^2 \big(\Delta B^{i_1}_r\big)^2  \Big]\mathbb{E}\Big[  \big(\Delta B^{i_2}_l\big)  \big(\Delta B^{i_2}_r\big) \Big]\\
		=&\sum_{l=j+1}^{k}\sum_{r=j+1}^{k}
		\bigg(  \mathbb{E} \Big[  \big(\Delta B^{i_1}_l\big) ^2 \Big]  \mathbb{E} \Big[  \big(\Delta B^{i_1}_r\big) ^2 \Big] +2 \Big(\mathbb{E}\Big[  \big(\Delta B^{i_1}_l\big) \big(\Delta B^{i_1}_r\big) \Big] \Big)^2 \bigg)
		\mathbb{E}\Big[  \big(\Delta B^{i_2}_l\big)  \big(\Delta B^{i_2}_r\big)\Big] \\
		=&\sum_{l=j+1}^{k}\sum_{r=j+1}^{k}
		\left( h^{4H} +2 \Big(\mathbb{E}\Big[  \big(\Delta B^{i_1}_l\big) \big(\Delta B^{i_1}_r\big) \Big] \Big)^2\right)
		\mathbb{E}\left[  \big(\Delta B^{i_2}_l\big)  \big(\Delta B^{i_2}_r\big) \right]\\
		=&|t-s|^{2H}h^{4H}+2\sum_{l=j+1}^{k}\sum_{r=j+1}^{k}\Big(\mathbb{E}\Big[  \big(\Delta B^{i_1}_l\big) \big(\Delta B^{i_1}_r\big) \Big] \Big)^3
		\le Ch^{6H-1}|t-s|,
	\end{align*}
	where $|t-s|\ge h$ and $H<\frac12$ are used in the last inequality. 
	Collecting the above  estimates, we conclude the statement.
\end{proof}

\begin{proposition}\label{prop-2}
	Under the assumptions  in Theorem \ref{main-1} and the definition of $\tilde{x}^{n,\alpha}$ in \eqref{x-alpha}, then for any $0<\beta<H$ and $0<\epsilon<2H-\frac12$, there exists a random variable $G=G(\beta,\epsilon)$ independent of $h=\frac{T}{n}$ such that
	\begin{align*}
		&\sup_{0\leq s<t\leq T} \frac{\left|\int_{s}^{t}\int_{s}^{u_1}{\rm d}\tilde{x}_{u_2}^{n,\alpha^2}{\rm d}\tilde{x}^{n,\alpha^1}_{u_1}\right|}{|t-s|^{2\beta}}\le G h^{2H-1/2-\epsilon},\quad |\alpha^1|=1,~|\alpha^2|=3,\\
		&\sup_{0\leq s<t\leq T} \frac{\left|\int_{s}^{t}\int_{s}^{u_1}{\rm d}\tilde{x}_{u_2}^{n,\alpha^2}{\rm d}\tilde{x}^{n,\alpha^1}_{u_1}\right|}{|t-s|^{2\beta}}\le G h^{2H-1/2-\epsilon},\quad |\alpha^1|=3,~|\alpha^2|=1,\\
		&\sup_{0\leq s<t\leq T} \frac{\left|\int_{s}^{t}\int_{s}^{u_1}{\rm d}\tilde{x}_{u_2}^{n,\alpha^2}{\rm d}\tilde{x}^{n,\alpha^1}_{u_1}\right|}{|t-s|^{2\beta}}\le G h^{2H-1/2-\epsilon},\quad |\alpha^1|=|\alpha^2|=3.
	\end{align*}
\end{proposition}
\begin{proof}
	Denote $\alpha^1=(\alpha^1_1,\cdots,\alpha^1_d)$ and $\alpha^2=(\alpha^2_1,\cdots,\alpha^2_d)$.
	Due to the upper bounds $|t-s|\le T$ and $h\le 1$, it is essential to consider the case that $\epsilon$ and $H-\beta$ are sufficiently small in the proof.
	
	First, assume $t_j\le s<t\le t_{j+1}$. It holds for  $H-\frac14<\beta<H$ that 
	\begin{align*}
		\left|\int_{s}^{t}\int_{s}^{u_1}{\rm d}\tilde{x}_{u_2}^{n,\alpha^2}{\rm d}\tilde{x}^{n,\alpha^1}_{u_1}\right|=&\int_{s}^{t}\int_{s}^{u_1}{\rm d}u_2{\rm d}u_1\left|  (\Delta B^{1}_{j+1})^{\alpha^1_1+\alpha^2_1}\cdots(\Delta B^{d}_{j+1})^{\alpha^1_d+\alpha^2_d}\right|h^{-2}\\
		\le & C \frac{(t-s)^2}{2h^2} \left|  (\Delta B^{1}_{j+1})^{\alpha^1_1+\alpha^2_1}\cdots(\Delta B^{d}_{j+1})^{\alpha^1_d+\alpha^2_d}\right|\\
		\le & G \Big(\frac{t-s}{h}\Big)^{2\beta} h^{4\beta}\le G h^{2H-\frac12-\epsilon}|t-s|^{2\beta}.
	\end{align*}
	
	Second, assume $t_{j-1}<s\le t_j\le t<t_{j+1}$. Then for $H-\frac14<\beta<H$, we have 
	\begin{align*}
		\left|\int_{s}^{t}\int_{s}^{u_1}{\rm d}\tilde{x}_{u_2}^{n,\alpha^2}{\rm d}\tilde{x}^{n,\alpha^1}_{u_1}\right|\le &\left|\int_{s}^{t_j}\int_{s}^{u_1}{\rm d}\tilde{x}_{u_2}^{n,\alpha^2}{\rm d}\tilde{x}^{n,\alpha^1}_{u_1}\right|+\left|\int_{t_j}^{t}\int_{s}^{t_j}{\rm d}\tilde{x}_{u_2}^{n,\alpha^2}{\rm d}\tilde{x}^{n,\alpha^1}_{u_1}\right|+\left|\int_{t_j}^{t}\int_{t_j}^{u_1}{\rm d}\tilde{x}_{u_2}^{n,\alpha^2}{\rm d}\tilde{x}^{n,\alpha^1}_{u_1}\right|\\
		\le & Gh^{2H-\frac12-\epsilon}\Big(|t_j-s|^{2\beta}+|t-t_j|^{\beta}|t_j-s|^{\beta}+|t-t_j|^{2\beta}\Big)\\
		\le & Gh^{2H-\frac12-\epsilon}|t-s|^{2\beta}.
	\end{align*}
	
	Next, we assume $t_{j-1}<s\le t_j<t_k\le t<t_{k+1}$ and divide the domain of the integral into six parts by
	\begin{align*}
		\int_{s}^{t}\int_{s}^{u_1}&=\int_{s}^{t_j}\int_{s}^{u_1}+\int_{t_j}^{t_k}\int_{s}^{t_j}+\int_{t_j}^{t_k}\int_{t_j}^{[u_1/n]h}+\int_{t_j}^{t_k}\int_{[u_1/n]h}^{u_1}+\int_{t_k}^{t}\int_{s}^{t_k}+\int_{t_k}^{t}\int_{t_k}^{u_1}\\
		&=:\int_{D_1}+\int_{D_2}+\int_{D_3}+\int_{D_4}+\int_{D_5}+\int_{D_6},
	\end{align*}
	where $[u_1/n]$ denotes the integer part of $u_1/n$.
	
	Suppose $|\alpha^1|=1$ and $|\alpha^2|=3$. The facts $0\le |t_j-s|\le h$ and $0\le| t-t_k|\le h$ yield 
	\begin{align*}
		\bigg|\int_{D_1}{\rm d}\tilde{x}_{u_2}^{n,\alpha^2}{\rm d}\tilde{x}^{n,\alpha^1}_{u_1}\bigg|+
		\bigg|\int_{D_6}{\rm d}\tilde{x}_{u_2}^{n,\alpha^2}{\rm d}\tilde{x}^{n,\alpha^1}_{u_1}\bigg|\le Gh^{4\beta}.
	\end{align*}
	Based on Lemma \ref{lm-2} and Lemma \ref{lm3}, we have  for $0<\delta<\frac12$ that
	\begin{align*}
		& \bigg|\int_{D_2}{\rm d}\tilde{x}_{u_2}^{n,\alpha^2}{\rm d}\tilde{x}^{n,\alpha^1}_{u_1}\bigg|+
		\bigg|\int_{D_5}{\rm d}\tilde{x}_{u_2}^{n,\alpha^2}{\rm d}\tilde{x}^{n,\alpha^1}_{u_1}\bigg|\\
		\le & \bigg|\int_{t_j}^{t_k}\tilde{x}_{u_1}^{n,\alpha^1}\bigg|\bigg|\int_{s}^{t_j}{\rm d}\tilde{x}^{n,\alpha^2}_{u_2}\bigg|+
		\bigg|\int_{t_k}^{t}\tilde{x}_{u_1}^{n,\alpha^1}\bigg|\bigg|\int_{s}^{t_k}{\rm d}\tilde{x}^{n,\alpha^2}_{u_2}\bigg|\\
		\le & G|t_k-t_j|^\beta \Big(\frac{t_j-s}{h}\Big)h^{3\beta}+ G\Big(\frac{t-t_k}{h}\Big)h^\beta h^{3H-\frac12-\epsilon}|t_k-s|^{\frac12-\delta}\\
		\le & Gh^{3\beta}|t-s|^{\beta}+Gh^{3H+\beta-\frac12-\epsilon}|t-s|^{\frac12-\delta},
	\end{align*}
	and 
	\begin{align*}
		&\bigg|\int_{D_4}{\rm d}\tilde{x}_{u_2}^{n,\alpha^2}{\rm d}\tilde{x}^{n,\alpha^1}_{u_1}\bigg|
		\le  \sum_{l=j+1}^{k}\bigg|\int_{t_{l-1}}^{t_l}\int_{t_{l-1}}^{u_1}{\rm d}\tilde{x}_{u_2}^{n,\alpha^2}{\rm d}\tilde{x}^{n,\alpha^1}_{u_1}\bigg|
		\le G \Big(\frac{t_k-t_j}{h}\Big) h^{4\beta}
		\le G  h^{4H-1-\epsilon}|t-s|.
	\end{align*}
	For the part of $\int_{D_3}{\rm d}\tilde{x}_{u_2}^{n,\alpha^2}{\rm d}\tilde{x}^{n,\alpha^1}_{u_1}$, define $f_{t_j,t_k}=\int_{t_j}^{t_k}\int_{t_j}^{[u_1/n]h}{\rm d}\tilde{x}_{u_2}^{n,\alpha^2}{\rm d}\tilde{x}^{n,\alpha^1}_{u_1}$. Then $f_{t_j,t_{j+1}}=0$, $j=0,\cdots,n-1$, and 
	\begin{align*}
		& \big|f_{t_j,t_l}-f_{t_j,t_k}-f_{t_k,t_l}\big|\\
		= & \bigg|\int_{t_j}^{t_l}\int_{t_j}^{[u_1/n]h}{\rm d}\tilde{x}_{u_2}^{n,\alpha^2}{\rm d}\tilde{x}^{n,\alpha^1}_{u_1}-\int_{t_j}^{t_k}\int_{t_j}^{[u_1/n]h}{\rm d}\tilde{x}_{u_2}^{n,\alpha^2}{\rm d}\tilde{x}^{n,\alpha^1}_{u_1}-\int_{t_k}^{t_l}\int_{t_k}^{[u_1/n]h}{\rm d}\tilde{x}_{u_2}^{n,\alpha^2}{\rm d}\tilde{x}^{n,\alpha^1}_{u_1}\bigg|\\
		= &
		\bigg|\int_{t_k}^{t_l}\int_{t_j}^{t_k}{\rm d}\tilde{x}_{u_2}^{n,\alpha^2}{\rm d}\tilde{x}^{n,\alpha^1}_{u_1}\bigg|
		\le  \bigg|\int_{t_k}^{t_l}\tilde{x}_{u_1}^{n,\alpha^1}\bigg|\bigg|\int_{t_j}^{t_k}{\rm d}\tilde{x}^{n,\alpha^2}_{u_2}\bigg|\\
		\le & G|t_l-t_k|^\beta h^{3H-1/2-\epsilon}|t_k-t_j|^{1/2-\delta},\\
		\le & G h^{2H-1/2-\epsilon}|t_l-t_j|^{1/2+H+\beta-\delta},\quad j<k<l.
	\end{align*} 
	Taking $0<\delta<H-\beta$ and $\frac14<\beta<H$ such that $1/2+H+\beta-\delta>1$, we apply Lemma \ref{lm5} to derive
	\begin{align*}
		\bigg|\int_{D_3}{\rm d}\tilde{x}_{u_2}^{n,\alpha^2}{\rm d}\tilde{x}^{n,\alpha^1}_{u_1}\bigg|=\big|f_{t_j,t_k}\big|
		\le  Gh^{2H-1/2-\epsilon}|t-s|^{1/2+H+\beta-\delta}.
	\end{align*}
	Using $|t-s|\ge h$, we obtain 
	\begin{align}\label{est-1}
		\bigg|\int_{t_j}^{t_k}\int_{t_j}^{u_1}{\rm d}\tilde{x}_{u_2}^{n,\alpha^2}{\rm d}\tilde{x}^{n,\alpha^1}_{u_1}\bigg|\le \bigg|\int_{D_3}{\rm d}\tilde{x}_{u_2}^{n,\alpha^2}{\rm d}\tilde{x}^{n,\alpha^1}_{u_1}\bigg|+\bigg|\int_{D_4}{\rm d}\tilde{x}_{u_2}^{n,\alpha^2}{\rm d}\tilde{x}^{n,\alpha^1}_{u_1}\bigg|
		\le Gh^{2H-1/2-\epsilon}|t-s|.
	\end{align}
	The above estimates produce
	\begin{align*}
		\sup_{0\leq s<t\leq T} \frac{\left|\int_{s}^{t}\int_{s}^{u_1}{\rm d}\tilde{x}_{u_2}^{n,\alpha^2}{\rm d}\tilde{x}^{n,\alpha^1}_{u_1}\right|}{|t-s|^{2\beta}}\le G h^{2H-1/2-\epsilon},\quad |\alpha^1|=1,~|\alpha^2|=3.
	\end{align*}
	Similarly, we have
	\begin{align*}
		\sup_{0\leq s<t\leq T} \frac{\left|\int_{s}^{t}\int_{s}^{u_1}{\rm d}\tilde{x}_{u_2}^{n,\alpha^2}{\rm d}\tilde{x}^{n,\alpha^1}_{u_1}\right|}{|t-s|^{2\beta}}\le G h^{2H-1/2-\epsilon},\quad |\alpha^1|=3,~|\alpha^2|=1.
	\end{align*}
	
	As to the case $|\alpha^1|=3$ and $|\alpha^2|=3$, repeating the techniques above, we have 
	\begin{align*}
		&\bigg|\int_{D_1}{\rm d}\tilde{x}_{u_2}^{n,\alpha^2}{\rm d}\tilde{x}^{n,\alpha^1}_{u_1}\bigg|+
		\bigg|\int_{D_6}{\rm d}\tilde{x}_{u_2}^{n,\alpha^2}{\rm d}\tilde{x}^{n,\alpha^1}_{u_1}\bigg|\le Gh^{6\beta},\\
		&\bigg|\int_{D_2}{\rm d}\tilde{x}_{u_2}^{n,\alpha^2}{\rm d}\tilde{x}^{n,\alpha^1}_{u_1}\bigg|+
		\bigg|\int_{D_5}{\rm d}\tilde{x}_{u_2}^{n,\alpha^2}{\rm d}\tilde{x}^{n,\alpha^1}_{u_1}\bigg|\le Gh^{6H-\frac12-\epsilon}|t-s|^{\frac12-\delta},\quad 0<\delta<\frac12,\\
		&\bigg|\int_{D_4}{\rm d}\tilde{x}_{u_2}^{n,\alpha^2}{\rm d}\tilde{x}^{n,\alpha^1}_{u_1}\bigg|\le G  h^{6H-1-\epsilon}|t-s|.
	\end{align*}
	Moreover, using 
	\begin{align*}
		\big|f_{t_j,t_l}-f_{t_j,t_k}-f_{t_k,t_l}\big|\le  \bigg|\int_{t_k}^{t_l}\tilde{x}_{u_1}^{n,\alpha^1}\bigg|\bigg|\int_{t_j}^{t_k}{\rm d}\tilde{x}^{n,\alpha^2}_{u_2}\bigg|
		\le Gh^{6H-1-\frac{\epsilon}{2}}|t-s|^{1-2\delta},\quad j<k<l,
	\end{align*} 
	and letting $0<\delta<\frac{\epsilon}{4}$ so that $1-2\delta+\frac{\epsilon}{2}>1$, we deduce from Lemma \ref{lm5} that
	\begin{align*}
		\bigg|\int_{D_3}{\rm d}\tilde{x}_{u_2}^{n,\alpha^2}{\rm d}\tilde{x}^{n,\alpha^1}_{u_1}\bigg|
		=\big|f_{t_j,t_k}\big|
		\le  Gh^{6H-1-\epsilon}|t-s|^{1-2\delta+\frac{\epsilon}{2}}.
	\end{align*}
	Therefore, we have
	\begin{align}\label{est-2}
		\bigg|\int_{t_j}^{t_k}\int_{t_j}^{u_1}{\rm d}\tilde{x}_{u_2}^{n,\alpha^2}{\rm d}\tilde{x}^{n,\alpha^1}_{u_1}\bigg|\le \bigg|\int_{D_3}{\rm d}\tilde{x}_{u_2}^{n,\alpha^2}{\rm d}\tilde{x}^{n,\alpha^1}_{u_1}\bigg|+\bigg|\int_{D_4}{\rm d}\tilde{x}_{u_2}^{n,\alpha^2}{\rm d}\tilde{x}^{n,\alpha^1}_{u_1}\bigg|
		\le Gh^{6H-1-\epsilon}|t-s|.
	\end{align}	
	These inequalities imply
	\begin{align*}
		\sup_{0\leq s<t\leq T} \frac{\left|\int_{s}^{t}\int_{s}^{u_1}{\rm d}\tilde{x}_{u_2}^{n,\alpha^2}{\rm d}\tilde{x}^{n,\alpha^1}_{u_1}\right|}{|t-s|^{2\beta}}\le G h^{2H-1/2-\epsilon},\quad |\alpha^1|=3,~|\alpha^2|=3.
	\end{align*}
\end{proof}

\begin{proposition}\label{prop-3}
	Under the assumptions  in Theorem \ref{main-1} and the definition of $\tilde{x}^{n,\alpha}$ in \eqref{x-alpha}, then for any $0<\beta<H$ and $0<\epsilon<2H-\frac12$, there exists a random variable $G=G(\beta,\epsilon)$ independent of $h=\frac{T}{n}$ such that
	\begin{align*}
		&\sup_{0\leq s<t\leq T} \frac{\left|\int_{s}^{t}\int_{s}^{u_1}\int_{s}^{u_2}{\rm d}\tilde{x}_{u_3}^{n,\alpha^3}{\rm d}\tilde{x}_{u_2}^{n,\alpha^2}{\rm d}\tilde{x}^{n,\alpha^1}_{u_1}\right|}{|t-s|^{3\beta}}\le G h^{2H-1/2-\epsilon},\quad |\alpha^1|=|\alpha^2|=|\alpha^3|=3,\\
		&\sup_{0\leq s<t\leq T} \frac{\left|\int_{s}^{t}\int_{s}^{u_1}\int_{s}^{u_2}{\rm d}\tilde{x}_{u_3}^{n,\alpha^3}{\rm d}\tilde{x}_{u_2}^{n,\alpha^2}{\rm d}\tilde{x}^{n,\alpha^1}_{u_1}\right|}{|t-s|^{3\beta}}\le G h^{2H-1/2-\epsilon},\quad |\alpha^1|=|\alpha^2|=3,~|\alpha^3|=1,\\
		&\sup_{0\leq s<t\leq T} \frac{\left|\int_{s}^{t}\int_{s}^{u_1}\int_{s}^{u_2}{\rm d}\tilde{x}_{u_3}^{n,\alpha^3}{\rm d}\tilde{x}_{u_2}^{n,\alpha^2}{\rm d}\tilde{x}^{n,\alpha^1}_{u_1}\right|}{|t-s|^{3\beta}}\le G h^{2H-1/2-\epsilon},\quad |\alpha^2|=|\alpha^3|=3,~|\alpha^1|=1,\\
		&\sup_{0\leq s<t\leq T} \frac{\left|\int_{s}^{t}\int_{s}^{u_1}\int_{s}^{u_2}{\rm d}\tilde{x}_{u_3}^{n,\alpha^3}{\rm d}\tilde{x}_{u_2}^{n,\alpha^2}{\rm d}\tilde{x}^{n,\alpha^1}_{u_1}\right|}{|t-s|^{3\beta}}\le G h^{2H-1/2-\epsilon},\quad |\alpha^1|=|\alpha^3|=3,~|\alpha^2|=1,\\
		&\sup_{0\leq s<t\leq T} \frac{\left|\int_{s}^{t}\int_{s}^{u_1}\int_{s}^{u_2}{\rm d}\tilde{x}_{u_3}^{n,\alpha^3}{\rm d}\tilde{x}_{u_2}^{n,\alpha^2}{\rm d}\tilde{x}^{n,\alpha^1}_{u_1}\right|}{|t-s|^{3\beta}}\le G h^{2H-1/2-\epsilon},\quad |\alpha^1|=3,~|\alpha^2|=|\alpha^3|=1,\\
		&\sup_{0\leq s<t\leq T} \frac{\left|\int_{s}^{t}\int_{s}^{u_1}\int_{s}^{u_2}{\rm d}\tilde{x}_{u_3}^{n,\alpha^3}{\rm d}\tilde{x}_{u_2}^{n,\alpha^2}{\rm d}\tilde{x}^{n,\alpha^1}_{u_1}\right|}{|t-s|^{3\beta}}\le G h^{2H-1/2-\epsilon},\quad |\alpha^2|=3,~|\alpha^1|=|\alpha^3|=1,\\
		&\sup_{0\leq s<t\leq T} \frac{\left|\int_{s}^{t}\int_{s}^{u_1}\int_{s}^{u_2}{\rm d}\tilde{x}_{u_3}^{n,\alpha^3}{\rm d}\tilde{x}_{u_2}^{n,\alpha^2}{\rm d}\tilde{x}^{n,\alpha^1}_{u_1}\right|}{|t-s|^{3\beta}}\le G h^{2H-1/2-\epsilon},\quad |\alpha^3|=3,~|\alpha^1|=|\alpha^2|=1.
	\end{align*}
\end{proposition}
\begin{proof}
	According to smilar arguments as in Propositions \ref{prop-1}-\ref{prop-2}, without loss of generality, we regard $\epsilon$ and $H-\beta$ as sufficiently small parameters in the proof, and focus on the case  $t_{j-1}<s\le t_j< t_k\le t<t_{k+1}$. 
	
	
	If $(|\alpha^1|,|\alpha^2|,|\alpha^3|)\in\{(3,3,3),(3,3,1),(1,3,3),(3,1,3),(3,1,1),(1,3,1)\}$, we decompose the domain of the integral into
	\begin{align*}
		\int_{s}^{t}\int_{s}^{u_1}\int_{s}^{u_2}=&\int_{s}^{t_j}\int_{s}^{u_1}\int_{s}^{u_2}+
		\int_{t_j}^{t_k}\int_{s}^{t_j}\int_{s}^{u_2}+
		\int_{t_j}^{t_k}\int_{t_j}^{[u_1/n]h}\int_{t_j}^{u_2}\\
		&+\int_{t_j}^{t_k}\int_{[u_1/n]h}^{u_1}\int_{t_j}^{u_2}+
		\int_{t_j}^{t_k}\int_{t_j}^{u_1}\int_{s}^{t_j}+
		\int_{t_k}^{t}\int_{s}^{u_1}\int_{s}^{u_2}\\
		=&:\int_{D_1}+\int_{D_2}+\int_{D_3}+\int_{D_4}+\int_{D_5}+\int_{D_6},
	\end{align*}
	where $[u_1/n]$ denotes the integer part of $u_1/n$. Let $H-\frac14<\beta<H$. Then we get 
	\begin{align*}
		&\left|\int_{D_1}{\rm d}\tilde{x}_{u_3}^{n,\alpha^3}{\rm d}\tilde{x}_{u_2}^{n,\alpha^2}{\rm d}\tilde{x}^{n,\alpha^1}_{u_1}\right|\le Gh^{5\beta}\le G h^{2H-1/2-\epsilon}|t-s|^{3\beta}.
	\end{align*}
	Together with Propositions \ref{prop-1}-\ref{prop-2} and Lemma \ref{lm-2}, we deduce
	\begin{align*}
		\left|\int_{D_2}{\rm d}\tilde{x}_{u_3}^{n,\alpha^3}{\rm d}\tilde{x}_{u_2}^{n,\alpha^2}{\rm d}\tilde{x}^{n,\alpha^1}_{u_1}\right|= &
		\bigg|\int_{t_j}^{t_k}{\rm d}\tilde{x}^{n,\alpha^1}_{u_1}\bigg|
		\bigg|  \int_{s}^{t_j}\int_{s}^{u_2}  {\rm d}\tilde{x}_{u_3}^{n,\alpha^3}{\rm d}\tilde{x}_{u_2}^{n,\alpha^2}\bigg|\\
		\le &Gh^{2H-1/2-\epsilon}|t_k-t_j|^{\beta}|t_j-s|^{2\beta}\\
		\le &Gh^{2H-1/2-\epsilon}|t-s|^{3\beta},
	\end{align*}
	and
	\begin{align*}
		\left|\int_{D_5}{\rm d}\tilde{x}_{u_3}^{n,\alpha^3}{\rm d}\tilde{x}_{u_2}^{n,\alpha^2}{\rm d}\tilde{x}^{n,\alpha^1}_{u_1}\right|
		= &
		\bigg|  \int_{t_j}^{t_k}\int_{t_j}^{u_1}  {\rm d}\tilde{x}_{u_2}^{n,\alpha^2}{\rm d}\tilde{x}_{u_1}^{n,\alpha^1}\bigg|\bigg|\int_{s}^{t_j}{\rm d}\tilde{x}^{n,\alpha^3}_{u_3}\bigg|\\
		\le &Gh^{2H-1/2-\epsilon}|t_k-t_j|^{2\beta}|t_j-s|^{\beta}\\
		\le &Gh^{2H-1/2-\epsilon}|t-s|^{3\beta}.
	\end{align*}
	Meanwhile, it holds that 
	\begin{align*}
		\left|\int_{D_6}{\rm d}\tilde{x}_{u_3}^{n,\alpha^3}{\rm d}\tilde{x}_{u_2}^{n,\alpha^2}{\rm d}\tilde{x}^{n,\alpha^1}_{u_1}\right|
		\le &
		\int_{t_k}^{t}\sup_{t_k<u_1<t}\bigg|\int_{s}^{u_1}\int_{s}^{u_2}{\rm d}\tilde{x}_{u_3}^{n,\alpha^3}{\rm d}\tilde{x}_{u_2}^{n,\alpha^2}\bigg|{\rm d}\big|\tilde{x}^{n,\alpha^1}_{u_1}\big|\\
		\le &
		\left\{
		\begin{aligned}
			&G|t-s|^{2\beta}\Big(\frac{t-t_k}{h}\Big)h^{3\beta},\qquad (|\alpha^1|,|\alpha^2|,|\alpha^3|)=(3,1,1)\\
			&Gh^{2H-1/2-\epsilon}|t-s|^{2\beta}\Big(\frac{t-t_k}{h}\Big)h^\beta,\quad {\rm else}\\
		\end{aligned}
		\right.\\
		\le &Gh^{2H-1/2-\epsilon}|t-s|^{3\beta}.
	\end{align*}	
	To estimate $\int_{D_4}{\rm d}\tilde{x}_{u_3}^{n,\alpha^3}{\rm d}\tilde{x}_{u_2}^{n,\alpha^2}{\rm d}\tilde{x}^{n,\alpha^1}_{u_1}$,  noticing $[u_2/n]=[u_1/n]$ for $[u_1/n]h<u_2<u_1$, we consider
	\begin{align*}
		&\left|\int_{D_4}{\rm d}\tilde{x}_{u_3}^{n,\alpha^3}{\rm d}\tilde{x}_{u_2}^{n,\alpha^2}{\rm d}\tilde{x}^{n,\alpha^1}_{u_1}\right|\\
		\le &
		\left|\int_{t_j}^{t_k}\int_{[u_1/n]h}^{u_1}\int_{t_j}^{[u_1/n]h}{\rm d}\tilde{x}_{u_3}^{n,\alpha^3}{\rm d}\tilde{x}_{u_2}^{n,\alpha^2}{\rm d}\tilde{x}^{n,\alpha^1}_{u_1}\right|+
		\left|\int_{t_j}^{t_k}\int_{[u_1/n]h}^{u_1}\int_{[u_1/n]h}^{u_2}{\rm d}\tilde{x}_{u_3}^{n,\alpha^3}{\rm d}\tilde{x}_{u_2}^{n,\alpha^2}{\rm d}\tilde{x}^{n,\alpha^1}_{u_1}\right|.
	\end{align*}
	For any $\frac12H+\frac{1}{8}<\beta<H$, we get
	\begin{align*}
		&\left|\int_{t_j}^{t_k}\int_{[u_1/n]h}^{u_1}\int_{t_j}^{[u_1/n]h}{\rm d}\tilde{x}_{u_3}^{n,\alpha^3}{\rm d}\tilde{x}_{u_2}^{n,\alpha^2}{\rm d}\tilde{x}^{n,\alpha^1}_{u_1}\right|\\
		\le&\sum_{l=j+1}^{k}\int_{t_{l-1}}^{t_l}\sup_{t_{l-1}<u_1<t_l}\bigg|\int_{[u_1/n]h}^{u_1}{\rm d}\tilde{x}_{u_2}^{n,\alpha^2}\bigg|\bigg|\int_{t_j}^{[u_1/n]h}{\rm d}\tilde{x}_{u_3}^{n,\alpha^3}\bigg|{\rm d}\big|\tilde{x}^{n,\alpha^1}_{u_1}\big|\\
		\le &
		G\Big(\frac{t_k-t_j}{h}\Big)h^{4\beta}\big|t_k -t_j\big|^\beta \le Gh^{2H-\frac12-\epsilon}|t-s|,
	\end{align*}
	and
	\begin{align*}
		&\left|\int_{t_j}^{t_k}\int_{[u_1/n]h}^{u_1}\int_{[u_1/n]h}^{u_2}{\rm d}\tilde{x}_{u_3}^{n,\alpha^3}{\rm d}\tilde{x}_{u_2}^{n,\alpha^2}{\rm d}\tilde{x}^{n,\alpha^1}_{u_1}\right|\\
		\le&\sum_{l=j+1}^{k}\int_{t_{l-1}}^{t_l}\sup_{t_{l-1}<u_1<t_l}\bigg|\int_{[u_1/n]h}^{u_1}\int_{[u_1/n]h}^{u_2}{\rm d}\tilde{x}_{u_3}^{n,\alpha^3}{\rm d}\tilde{x}_{u_2}^{n,\alpha^2}\bigg|{\rm d}\big|\tilde{x}^{n,\alpha^1}_{u_1}\big|\\
		\le &
		G\Big(\frac{t_k-t_j}{h}\Big)h^{5\beta}\le h^{2H-\frac12-\epsilon}|t-s|.
	\end{align*}
	Gathering the above two estimates, we obtain 
	\begin{align*}
		\left|\int_{D_4}{\rm d}\tilde{x}_{u_3}^{n,\alpha^3}{\rm d}\tilde{x}_{u_2}^{n,\alpha^2}{\rm d}\tilde{x}^{n,\alpha^1}_{u_1}\right|\le 
		Gh^{2H-\frac12-\epsilon}|t-s|^{3\beta}.
	\end{align*}
	To estimate $\int_{D_3}{\rm d}\tilde{x}_{u_3}^{n,\alpha^3}{\rm d}\tilde{x}_{u_2}^{n,\alpha^2}{\rm d}\tilde{x}^{n,\alpha^1}_{u_1}$, we define
	$f_{t_j,t_k}=\int_{t_j}^{t_k}\int_{t_j}^{[u_1/n]h}\int_{t_j}^{u_2}{\rm d}\tilde{x}_{u_3}^{n,\alpha^3}{\rm d}\tilde{x}_{u_2}^{n,\alpha^2}{\rm d}\tilde{x}^{n,\alpha^1}_{u_1}$ such that $f_{t_j,t_{j+1}}=0$ and for integers $j<k<l$,
	\begin{align*}
		&\big|f_{t_j,t_l}-f_{t_j,t_k}-f_{t_k,t_l}\big|\\
		\le &\bigg|\int_{t_k}^{t_l}\int_{t_j}^{t_k}\int_{t_j}^{u_2}{\rm d}\tilde{x}_{u_3}^{n,\alpha^3}{\rm d}\tilde{x}_{u_2}^{n,\alpha^2}{\rm d}\tilde{x}^{n,\alpha^1}_{u_1}\bigg|
		+\bigg|\int_{t_k}^{t_l}\int_{t_k}^{[u_1/n]h}\int_{t_j}^{t_k}{\rm d}\tilde{x}_{u_3}^{n,\alpha^3}{\rm d}\tilde{x}_{u_2}^{n,\alpha^2}{\rm d}\tilde{x}^{n,\alpha^1}_{u_1}\bigg|\\
		\le&\bigg|\int_{t_k}^{t_l}\int_{t_j}^{t_k}\int_{t_j}^{u_2}{\rm d}\tilde{x}_{u_3}^{n,\alpha^3}{\rm d}\tilde{x}_{u_2}^{n,\alpha^2}{\rm d}\tilde{x}^{n,\alpha^1}_{u_1}\bigg|\\ &+\bigg|\int_{t_k}^{t_l}\int_{t_k}^{u_1}\int_{t_j}^{t_k}{\rm d}\tilde{x}_{u_3}^{n,\alpha^3}{\rm d}\tilde{x}_{u_2}^{n,\alpha^2}{\rm d}\tilde{x}^{n,\alpha^1}_{u_1}-\int_{t_k}^{t_l}\int_{[u_1/n]h}^{u_1}\int_{t_j}^{t_k}{\rm d}\tilde{x}_{u_3}^{n,\alpha^3}{\rm d}\tilde{x}_{u_2}^{n,\alpha^2}{\rm d}\tilde{x}^{n,\alpha^1}_{u_1}\bigg|\\
		\le&\bigg|\int_{t_k}^{t_l}{\rm d}\tilde{x}^{n,\alpha^1}_{u_1}\bigg|\bigg|\int_{t_j}^{t_k}\int_{t_j}^{u_2}{\rm d}\tilde{x}_{u_3}^{n,\alpha^3}{\rm d}\tilde{x}_{u_2}^{n,\alpha^2}\bigg|\\ &+\bigg|\int_{t_k}^{t_l}\int_{t_k}^{u_1}{\rm d}\tilde{x}_{u_2}^{n,\alpha^2}{\rm d}\tilde{x}^{n,\alpha^1}_{u_1}\bigg|\bigg|\int_{t_j}^{t_k}{\rm d}\tilde{x}_{u_3}^{n,\alpha^3}\bigg|+\bigg|\int_{t_k}^{t_l}\int_{[u_1/n]h}^{u_1}{\rm d}\tilde{x}_{u_2}^{n,\alpha^2}{\rm d}\tilde{x}^{n,\alpha^1}_{u_1}\bigg|\bigg|\int_{t_j}^{t_k}{\rm d}\tilde{x}_{u_3}^{n,\alpha^3}\bigg|.
	\end{align*}
	Lemma \ref{lm-2} and \eqref{est-1}-\eqref{est-2} yield
	\begin{align*}
		&\bigg|\int_{t_k}^{t_l}{\rm d}\tilde{x}^{n,\alpha^1}_{u_1}\bigg|\bigg|\int_{t_j}^{t_k}\int_{t_j}^{u_2}{\rm d}\tilde{x}_{u_3}^{n,\alpha^3}{\rm d}\tilde{x}_{u_2}^{n,\alpha^2}\bigg| +\bigg|\int_{t_k}^{t_l}\int_{t_k}^{u_1}{\rm d}\tilde{x}_{u_2}^{n,\alpha^2}{\rm d}\tilde{x}^{n,\alpha^1}_{u_1}\bigg|\bigg|\int_{t_j}^{t_k}{\rm d}\tilde{x}_{u_3}^{n,\alpha^3}\bigg|\\
		\le 	&
		\left\{
		\begin{aligned}
			&Gh^{3H-1/2-\epsilon}|t_l-t_k|^{\frac12-\delta}|t_k-t_j|^{2\beta}+Gh^{2H-1/2-\epsilon}|t_l-t_k||t_k-t_j|^{\beta},\quad(|\alpha^1|,|\alpha^2|,|\alpha^3|)=(3,1,1)\\
			&G|t_l-t_k|^\beta h^{2H-1/2-\epsilon}|t_k-t_j|+Gh^{2H-1/2-\epsilon}|t_l-t_k||t_k-t_j|^\beta,\quad {\rm else}\\
		\end{aligned}
		\right.
	\end{align*}
	and 
	\begin{align*}
		\bigg|\int_{t_k}^{t_l}\int_{[u_1/n]h}^{u_1}{\rm d}\tilde{x}_{u_2}^{n,\alpha^2}{\rm d}\tilde{x}^{n,\alpha^1}_{u_1}\bigg|\bigg|\int_{t_j}^{t_k}{\rm d}\tilde{x}_{u_3}^{n,\alpha^3}\bigg|\le G\Big(\frac{t_l-t_k}{h}\Big)h^{4\beta} |t_k-t_j|^\beta.
	\end{align*}
	Then for $\frac{H}{2}+\frac18<\beta<H$ and $0<\delta<2(\beta-\frac14)$ such that $\frac12+2\beta-\delta>1$, Lemma \ref{lm5}  leads to
	\begin{align*}
		\bigg|\int_{D_3}{\rm d}\tilde{x}_{u_3}^{n,\alpha^3}{\rm d}\tilde{x}_{u_2}^{n,\alpha^2}{\rm d}\tilde{x}^{n,\alpha^1}_{u_1}\bigg|=\big|f_{t_j,t_k}\big|\le  Gh^{2H-\frac12-\epsilon}|t-s|^{\frac12+2\beta-\delta}.
	\end{align*}
	Hence, 
	\begin{align*}
		\bigg|\int_{s}^{t}\int_{s}^{u_1}\int_{s}^{u_2}{\rm d}\tilde{x}_{u_3}^{n,\alpha^3}{\rm d}\tilde{x}_{u_2}^{n,\alpha^2}{\rm d}\tilde{x}^{n,\alpha^1}_{u_1}\bigg|\le  Gh^{2H-\frac12-\epsilon}|t-s|^{3\beta}.
	\end{align*}
	
	If  $(|\alpha^1|,|\alpha^2|,|\alpha^3|)=(1,1,3)$, based on 
	\begin{align*}
		\int_{s}^{t}\int_{s}^{u_1}\int_{s}^{u_2}{\rm d}\tilde{x}_{u_3}^{n,\alpha^3}{\rm d}\tilde{x}_{u_2}^{n,\alpha^2}{\rm d}\tilde{x}^{n,\alpha^1}_{u_1}=\int_{s}^{t}\int_{u_3}^{t}\int_{u_2}^{t}{\rm d}\tilde{x}^{n,\alpha^1}_{u_1}{\rm d}\tilde{x}_{u_2}^{n,\alpha^2}{\rm d}\tilde{x}_{u_3}^{n,\alpha^3},
	\end{align*}
	the decomposition of the domain is changed to 
	\begin{align*}
		\int_{s}^{t}\int_{u_3}^{t}\int_{u_2}^{t}=&\int_{t_k}^{t}\int_{u_3}^{t}\int_{u_2}^{t}+
		\int_{t_j}^{t_k}\int_{t_k}^{t}\int_{u_2}^{t}+
		\int_{t_j}^{t_k}\int_{\lceil u_3/n\rceil h}^{t_k}\int_{u_2}^{t_k}\\
		&+\int_{t_j}^{t_k}\int_{u_3}^{\lceil u_3/n\rceil h}\int_{u_2}^{t_k}+
		\int_{t_j}^{t_k}\int_{u_3}^{t_k}\int_{t_k}^{t}+
		\int_{s}^{t_j}\int_{u_3}^{t}\int_{u_2}^{t}\\
		=&:\int_{\tilde{D}_1}+\int_{\tilde{D}_2}+\int_{\tilde{D}_3}+\int_{\tilde{D}_4}+\int_{\tilde{D}_5}+\int_{\tilde{D}_6},
	\end{align*}
	where $\lceil u_3/n\rceil=[u_3/n]+1$. Then we get for $H-\frac14<\beta<H$ and $0<\delta<\frac12-\beta$,
	\begin{align*}
		\left|\int_{\tilde{D}_1}{\rm d}\tilde{x}_{u_1}^{n,\alpha^1}{\rm d}\tilde{x}_{u_2}^{n,\alpha^2}{\rm d}\tilde{x}^{n,\alpha^3}_{u_3}\right|\le& Gh^{5\beta}\le G h^{2H-1/2-\epsilon}|t-s|^{3\beta},\\
		\left|\int_{\tilde{D}_2}{\rm d}\tilde{x}_{u_1}^{n,\alpha^1}{\rm d}\tilde{x}_{u_2}^{n,\alpha^2}{\rm d}\tilde{x}^{n,\alpha^3}_{u_3}\right|=&
		\bigg|\int_{t_j}^{t_k}{\rm d}\tilde{x}^{n,\alpha^3}_{u_3}\bigg|
		\bigg|  \int_{t_k}^{t}\int_{t_k}^{u_1}  {\rm d}\tilde{x}_{u_2}^{n,\alpha^2}{\rm d}\tilde{x}_{u_1}^{n,\alpha^1}\bigg|\\
		\le &Gh^{3H-1/2-\epsilon}|t_k-t_j|^{\frac12-\delta}|t-t_k|^{2\beta}
		\le Gh^{2H-1/2-\epsilon}|t-s|^{3\beta},\\
		\left|\int_{\tilde{D}_5}{\rm d}\tilde{x}_{u_1}^{n,\alpha^1}{\rm d}\tilde{x}_{u_2}^{n,\alpha^2}{\rm d}\tilde{x}^{n,\alpha^3}_{u_3}\right|
		= &
		\bigg|  \int_{t_j}^{t_k}\int_{t_j}^{u_2}  {\rm d}\tilde{x}_{u_23}^{n,\alpha^3}{\rm d}\tilde{x}_{u_2}^{n,\alpha^2}\bigg|\bigg|\int_{t_k}^{t}{\rm d}\tilde{x}^{n,\alpha^1}_{u_1}\bigg|\\
		\le &Gh^{2H-1/2-\epsilon}|t_k-t_j|^{2\beta}|t-t_k|^{\beta}
		\le Gh^{2H-1/2-\epsilon}|t-s|^{3\beta},\\
		\left|\int_{\tilde{D}_6}{\rm d}\tilde{x}_{u_1}^{n,\alpha^1}{\rm d}\tilde{x}_{u_2}^{n,\alpha^2}{\rm d}\tilde{x}^{n,\alpha^3}_{u_3}\right|
		\le &
		\int_{s}^{t_j}\sup_{s<u_3<t_j}\bigg|\int_{u_3}^{t}\int_{u_3}^{u_1}{\rm d}\tilde{x}_{u_2}^{n,\alpha^2}{\rm d}\tilde{x}_{u_1}^{n,\alpha^1}\bigg|{\rm d}\big|\tilde{x}^{n,\alpha^3}_{u_3}\big|\\
		\le & G|t-s|^{2\beta}\Big(\frac{t_j-s}{h}\Big)h^{3\beta}
		\le Gh^{2H-1/2-\epsilon}|t-s|^{3\beta}.
	\end{align*}	
	Furthermore, it follows from  $\lceil u_2/n\rceil=\lceil u_3/n\rceil$ for $u_3<u_2<\lceil u_3/n\rceil h$ that 
	\begin{align*}
		\left|\int_{\tilde{D}_4}{\rm d}\tilde{x}_{u_1}^{n,\alpha^1}{\rm d}\tilde{x}_{u_2}^{n,\alpha^2}{\rm d}\tilde{x}^{n,\alpha^3}_{u_3}\right|\le &
		\left|\int_{t_j}^{t_k}\int_{u_3}^{\lceil u_3/n\rceil h}\int_{\lceil u_3/n\rceil h}^{t_k}{\rm d}\tilde{x}_{u_1}^{n,\alpha^1}{\rm d}\tilde{x}_{u_2}^{n,\alpha^2}{\rm d}\tilde{x}^{n,\alpha^3}_{u_3}\right|\\
		&+
		\left|\int_{t_j}^{t_k}\int_{u_3}^{\lceil u_3/n\rceil h}\int_{u_2}^{\lceil u_3/n\rceil h}{\rm d}\tilde{x}_{u_1}^{n,\alpha^1}{\rm d}\tilde{x}_{u_2}^{n,\alpha^2}{\rm d}\tilde{x}^{n,\alpha^3}_{u_3}\right|\\
		\le&\sum_{l=j+1}^{k}\int_{t_{l-1}}^{t_l}\sup_{t_{l-1}<u_3<t_l}\bigg|\int_{u_3}^{\lceil u_3/n\rceil h}{\rm d}\tilde{x}_{u_2}^{n,\alpha^2}\bigg|\bigg|\int_{\lceil u_3/n\rceil h}^{t_k}{\rm d}\tilde{x}_{u_1}^{n,\alpha^1}\bigg|{\rm d}\big|\tilde{x}^{n,\alpha^3}_{u_3}\big|\\
		& +  \sum_{l=j+1}^{k}\int_{t_{l-1}}^{t_l}\sup_{t_{l-1}<u_3<t_l}\bigg|\int_{u_3}^{\lceil u_3/n\rceil h}\int_{u_2}^{\lceil u_3/n\rceil h}{\rm d}\tilde{x}_{u_1}^{n,\alpha^1}{\rm d}\tilde{x}_{u_2}^{n,\alpha^2}\bigg|{\rm d}\big|\tilde{x}^{n,\alpha^3}_{u_3}\big|\\
		\le &
		G\Big(\frac{t_k-t_j}{h}\Big)h^{4\beta}\big|t_k-t_j\big|^\beta+ G\Big(\frac{t_k-t_j}{h}\Big)h^{5\beta}\\
		\le &Gh^{2H-\frac12-\epsilon}|t-s|^{3\beta},
	\end{align*}
	where $\frac12H+\frac{1}{8}<\beta<H$.
	In order to estimate the last part, we define
	$$\tilde{f}_{t_j,t_k}=\int_{t_j}^{t_k}\int_{\lceil u_3/n\rceil h}^{t_k}\int_{u_2}^{t_k}{\rm d}\tilde{x}_{u_1}^{n,\alpha^1}{\rm d}\tilde{x}_{u_2}^{n,\alpha^2}{\rm d}\tilde{x}^{n,\alpha^3}_{u_3}$$ such that $\tilde{f}_{t_j,t_{j+1}}=0$ and for $j<k<l$,
	\begin{align*}
		&\big|\tilde{f}_{t_j,t_l}-\tilde{f}_{t_j,t_k}-\tilde{f}_{t_k,t_l}\big|\\
		\le & \bigg|\int_{t_j}^{t_k}\int_{u_3}^{t_k}\int_{t_k}^{t_l}{\rm d}\tilde{x}_{u_1}^{n,\alpha^1}{\rm d}\tilde{x}_{u_2}^{n,\alpha^2}{\rm d}\tilde{x}^{n,\alpha^3}_{u_3}\bigg|
		+ \bigg|\int_{t_j}^{t_k}\int_{t_k}^{t_l}\int_{u_2}^{t_l}{\rm d}\tilde{x}_{u_1}^{n,\alpha^1}{\rm d}\tilde{x}_{u_2}^{n,\alpha^2}{\rm d}\tilde{x}^{n,\alpha^3}_{u_3}\bigg|\\
		&+ \bigg|\int_{t_j}^{t_k}\int_{u_3}^{ \lceil u_3/n\rceil h  }\int_{t_k}^{t_l}{\rm d}\tilde{x}_{u_1}^{n,\alpha^1}{\rm d}\tilde{x}_{u_2}^{n,\alpha^2}{\rm d}\tilde{x}^{n,\alpha^3}_{u_3}\bigg|\\
		\le&\bigg|\int_{t_k}^{t_l}{\rm d}\tilde{x}^{n,\alpha^1}_{u_1}\bigg|\bigg|\int_{t_j}^{t_k}\int_{t_j}^{u_2}{\rm d}\tilde{x}_{u_3}^{n,\alpha^3}{\rm d}\tilde{x}_{u_2}^{n,\alpha^2}\bigg| +\bigg|\int_{t_k}^{t_l}\int_{t_k}^{u_1}{\rm d}\tilde{x}_{u_2}^{n,\alpha^2}{\rm d}\tilde{x}^{n,\alpha^1}_{u_1}\bigg|\bigg|\int_{t_j}^{t_k}{\rm d}\tilde{x}_{u_3}^{n,\alpha^3}\bigg|\\
		&+
		\bigg|\int_{t_j}^{t_k}\int_{u_3}^{\lceil u_3/n\rceil h  }{\rm d}\tilde{x}_{u_2}^{n,\alpha^2}{\rm d}\tilde{x}^{n,\alpha^3}_{u_3}\bigg|\bigg|\int_{t_k}^{t_l}{\rm d}\tilde{x}_{u_1}^{n,\alpha^1}\bigg|\\
		\le &G|t_l-t_k|^{\beta}h^{2H-1/2-\epsilon}|t_k-t_j|+G|t_l-t_k|^{2\beta}h^{3H-1/2-\epsilon}|t_k-t_j|^{\frac12-\delta}+G\Big(\frac{t_k-t_j}{h}\Big)h^{4\beta} |t_l-t_k|^\beta.
	\end{align*}
	Then for $\frac{H}{2}+\frac18<\beta<H$ and $0<\delta<2(\beta-\frac14)$ such that $\frac12+2\beta-\delta>1$, Lemma \ref{lm5}  leads to
	\begin{align*}
		\bigg|\int_{\tilde{D}_3}{\rm d}\tilde{x}_{u_3}^{n,\alpha^3}{\rm d}\tilde{x}_{u_2}^{n,\alpha^2}{\rm d}\tilde{x}^{n,\alpha^1}_{u_1}\bigg|=\big|\tilde{f}_{t_j,t_k}\big|\le  Gh^{2H-\frac12-\epsilon}|t-s|^{\frac12+2\beta-\delta}.
	\end{align*}
	As a consequence, 
	\begin{align*}
		\bigg|\int_{s}^{t}\int_{s}^{u_1}\int_{s}^{u_2}{\rm d}\tilde{x}_{u_3}^{n,\alpha^3}{\rm d}\tilde{x}_{u_2}^{n,\alpha^2}{\rm d}\tilde{x}^{n,\alpha^1}_{u_1}\bigg|\le  Gh^{2H-\frac12-\epsilon}|t-s|^{3\beta},\quad |\alpha^3|=3,~|\alpha^1|=|\alpha^2|=1,
	\end{align*}
	which finishes the proof.
\end{proof}

\bibliographystyle{plain}
\bibliography{bib}

\end{document}